\documentclass[final,onefignum,onetabnum]{siamart190516}

\usepackage{lipsum}
\usepackage{amsfonts}
\usepackage{graphicx}
\usepackage{epstopdf}

\usepackage{algorithm}
\usepackage{algorithmicx}
\usepackage{algpseudocode}

\ifpdf
  \DeclareGraphicsExtensions{.eps,.pdf,.png,.jpg}
\else
  \DeclareGraphicsExtensions{.eps}
\fi

\newsiamremark{remark}{Remark}
\newsiamremark{example}{Example}
\newsiamremark{hypothesis}{Hypothesis}
\crefname{hypothesis}{Hypothesis}{Hypotheses}
\newsiamthm{claim}{Claim}

\headers{Error bounds for Lanczos-FA}{T. Chen, A. Greenbaum, C. Musco, and C. Musco}

\title{Error bounds for Lanczos-based matrix function approximation%
\thanks{%Submitted to the editors on June 17, 2021.
\textbf{Funding:} This material is based on work supported by the National Science Foundation under Grant Nos. DGE-1762114, CCF-2045590, and CCF-2046235 and by an Adobe Research grant. Any opinions, findings, and conclusions or recommendations expressed in this material are those of the authors and do not necessarily reflect the views of the National Science Foundation.}}

\author{Tyler Chen\thanks{University of Washington, \href{mailto:chentyl@uw.edu}{\texttt{chentyl@uw.edu}}}
\and Anne Greenbaum\thanks{University of Washington, \href{mailto:greenbau_uw.edu}{\texttt{greenbau@uw.edu}}}
\and Cameron Musco\thanks{University of Massachusetts Amherst, \href{mailto:cmusco@cs.umass.edu}{\texttt{cmusco@cs.umass.edu}}} 
\and \linebreak[4] Christopher Musco\thanks{New York University, \href{mailto:cmusco@nyu.edu}{\texttt{cmusco@nyu.edu}}}
}

\usepackage{amsopn}

\crefname{section}{Section}{Sections}
\crefname{ineq}{}{}
\creflabelformat{ineq}{#2{\upshape(#1)}#3}

\usepackage{marginnote}
\usepackage{enumitem}

\usepackage{tikz}
\usetikzlibrary{calc}
\usepackage{graphicx}
\usepackage{subcaption}
\usepackage{booktabs}
\usepackage{array}

\usepackage{afterpage}
\usepackage{pdflscape}

\newcommand{\Tyler}[1]{\textcolor{red}{Tyler: #1}}
\newcommand{\Anne}[1]{\textcolor{green}{Anne: #1}}

\DeclareMathOperator*{\argmin}{argmin}

\usepackage{bbm}

\usepackage{mathtools}

\usepackage{bm}

\newcommand{\R}{\mathbb{R}}

\renewcommand{\Re}{\operatorname{Re}}
\renewcommand{\Im}{\operatorname{Im}}

\renewcommand{\d}[1]{\ensuremath{\mathrm{d}#1}}

\renewcommand{\vec}{\mathbf}

\newcommand{\T}{\textup{\textsf{T}}}
\newcommand{\cT}{\textup{\textsf{H}}}

\usepackage{xspace}

\newcommand{\lan}{\textup{\textsf{lan}}}
\newcommand{\err}{\textup{\textsf{err}}}
\newcommand{\Res}{\textup{\textsf{res}}}

\usepackage[normalem]{ulem}
\newcommand{\rev}[2]{{\color{red}\sout{}}{\color{blue}#2}}

\ifpdf
\hypersetup{
  pdftitle={Error bounds for Lanczos-based matrix function approximation},
  pdfauthor={T. Chen, A. Greenbaum, C. Musco, and C. Musco}
}
\fi

\begin{document}

\maketitle

\begin{abstract}
We analyze the Lanczos method for matrix function approximation (Lanczos-FA), an iterative algorithm for computing \( f(\vec{A}) \vec{b}\) when \(\vec{A}\) is a  Hermitian matrix and \(\vec{b}\) is a given vector. 
Assuming that \( f : \mathbb{C} \rightarrow \mathbb{C} \) is piecewise analytic, we give a framework, based on the Cauchy integral formula, which can be used to derive a priori and a posteriori error bounds for Lanczos-FA in terms of the error of Lanczos used to solve linear systems.
Unlike many error bounds for Lanczos-FA, these  bounds account for fine-grained properties of the spectrum of \( \vec{A} \), such as clustered or isolated eigenvalues.
Our results are derived assuming exact arithmetic, but we show that they are easily extended to finite precision computations using existing theory about the Lanczos algorithm in finite precision.
We also provide generalized bounds for the Lanczos method used to approximate quadratic forms \( \vec{b}^\cT f(\vec{A}) \vec{b} \), and demonstrate the effectiveness of our bounds with numerical experiments.
\end{abstract}

% REQUIRED

% REQUIRED

\section{Introduction}

Computing the product of a matrix function
%\footnote{Here \( f(\vec{A}) := \vec{U} f(\vec{\Lambda}) \vec{U}^{\cT} \) where \( \vec{A} = \vec{U} \vec{\Lambda} \vec{U}^{\cT} \) is an eigendecomposition of \( \vec{A} \), and \( f(\vec{\Lambda}) \) is simply \( f \) applied individually to the diagonal entries of the diagonal matrix \( \vec{\Lambda} \).}%
\( f(\vec{A}) \) with a  vector \( \vec{b} \), where \( \vec{A} \) is a Hermitian matrix and \( f : \mathbb{C} \to \mathbb{C} \) is a scalar function, is a fundamental task in numerical linear algebra.
Perhaps the most well known example is \( f(x) = 1/x \), in which case \( f(\vec{A}) \vec{b}  = \vec{A}^{-1} \vec{b} \) is the solution to the linear system of equations \( \vec{A} \vec{x} = \vec{b} \).
Other common functions include the exponential, logarithm, square root, inverse square root, and sign function, which have applications in solving  
differential equations \cite{druskin_knizhnerman_89, saad_92}, 
Gaussian process sampling \cite{pleiss_jankowiak_eriksson_damle_gardner_20},
principal component projection and regression \cite{allen_zhu_li_17,frostig_musco_musco_sidford_16,jin_sidford_19},
lattice quantum chromodynamics \cite{davies_higham_05,eshof_frommer_lippert_schilling_van_der_vorst_02},
eigenvalue counting/spectrum approximation \cite{braverman_krishnan_musco_21,chen_trogdon_ubaru_21,di_napoli_polizzi_saad_16}, and beyond \cite{higham_08}.

A common approach to approximating \( f(\vec{A})\vec{b} \) is based on the Lanczos algorithm.
The Lanczos algorithm, shown in \cref{alg:lanczos}, iteratively constructs an orthonormal basis \( \vec{Q}_k = [ \vec{q}_1 , \ldots , \vec{q}_k ]\) for a nested sequence of  Krylov subspaces,
\begin{align*}
\mathcal{K}_k(\vec{A},\vec{b})
= \operatorname{span}( \vec{b}, \vec{A}\vec{b}, \ldots, \vec{A}^{k-1}\vec{b} )
= \{ p(\vec{A}) \vec{b} : \deg(p) < k \},
%= \{ p(\vec{A}) \vec{b} : p \in \mathcal{P}_k \},
\end{align*}
such that \( \operatorname{span}(\vec{q}_1, \ldots, \vec{q}_j) = \mathcal{K}_j(\vec{A},\vec{b}) \) for all \( j\leq k\).
The basis \( \vec{Q}_k \) satisfies a three-term recurrence 
\begin{align}
    \label{eqn:lanczos_factorization}
    \vec{A} \vec{Q}_k = \vec{Q}_k \vec{T}_k + \beta_k \vec{q}_{k+1} \vec{e}_k^{\T},
\end{align}
where \( \vec{T}_k \) is a real symmetric tridiagonal matrix with entries
\begin{align*}
    \vec{T}_k
    = \begin{bmatrix}
        \alpha_1 & \beta_1 \\
        \beta_1 & \alpha_2 & \ddots \\
        & \ddots & \ddots & \beta_{k-1}\\
        & & \beta_{k-1} & \alpha_k
    \end{bmatrix}.
\end{align*}

The Lanczos method for matrix function approximation, which we refer to as Lanczos-FA, approximates \( f(\vec{A})\vec{b} \) using   \( \vec{Q}_k \) and \( \vec{T}_k \) as follows: 
\begin{definition}
The \( k \)-th Lanczos-FA approximation to \( f(\vec{A}) \vec{b} \) is defined as
\begin{align*}
    \lan_k(f,\vec{A},\vec{b}) 
    := \vec{Q}_k f(\vec{T}_k) \vec{Q}_k^{\cT} \vec{b},
\end{align*}
where \( \vec{Q}_k \) and \( \vec{T}_k \) are produced by the Lanczos method run for \( k \) steps on \( (\vec{A},\vec{b}) \).
For simplicity, we often write \( \lan_k(f) \), since \( \vec{A} \) and \( \vec{b} \) remain fixed for most of this manuscript.
If we are considering the Lanczos algorithm run on a matrix or right hand side different from the given \( \vec{A} \) or \( \vec{b} \), we will use the full notation.
\end{definition}

\begin{center}
\begin{minipage}[t]{.75\textwidth}%{.52\textwidth}
    \begin{algorithm}[H]
\caption{Lanczos}\label{alg:lanczos}
\fontsize{10}{10}\selectfont
\begin{algorithmic}[1]
\Procedure{Lanczos}{$\vec{A}, \vec{b}, k$}
\State \( \vec{q}_0 = \vec{0} \),
\( \beta_{0} = 0 \),
\( \vec{q}_1  = \vec{b} / \| \vec{b} \| \)
\For {\( j=1,2,\ldots,k \)}
    \State \( \tilde{\vec{q}}_{j+1} = \vec{A} \vec{q}_{j} - \beta_{j-1} \vec{q}_{j-1} \)
    \State \( \alpha_j = \langle \tilde{\vec{q}}_{j+1}, \vec{q}_j \rangle \)
    \State \( \tilde{\vec{q}}_{j+1} = \tilde{\vec{q}}_{j+1} - \alpha_j \vec{q}_j \)
%    \If{reorthogonalize}
        \State optionally, reorthogonalize\footnotemark\,  \( \tilde{\vec{q}}_{j+1} \) against \( \{\vec{q}_i\}_{i=1}^{j-1} \)
%    \EndIf
    \State \( \beta_{j} = \| \tilde{\vec{q}}_{j+1} \| \)
    \State \( \vec{q}_{j+1} = \tilde{\vec{q}}_{j+1} / \beta_{j} \)
\EndFor
\State \Return \( \vec{Q}_k \), \( \vec{T}_k \)
\EndProcedure
\end{algorithmic}
\end{algorithm}

\end{minipage}
\vspace{1.5em}
\end{center}

\footnotetext{Note that reorthogonalization has no effect on the algorithm in exact arithmetic, but can in finite precision. We discuss finite precision considerations in \cref{sec:finite_precision}.}

We would like to understand the convergence behavior of Lanczos-FA through a priori and a posteriori error bounds.
In the context of Krylov subspace methods for symmetric matrices, a priori bounds depend on the spectrum of \( \vec{A} \) but not on the choice of right hand side \( \vec{b} \) \cite{greenbaum_97}.
As such, a priori bounds are used to provide intuition about how an algorithm depends on the spectrum of the input.
On the other hand, a posteriori bounds typically depend on quantities which are accessible to the user, but not on quantities which are unknown in practice.
This means a posteriori bounds for Lanczos-FA can depend on quantities such as the output of the Lanczos algorithm $\vec{Q}_k$ and $\vec{T}_k$ but not on the spectrum of $\vec{A}$.

\subsection{Polynomial error bounds for Lanczos-FA}
\label{sec:polynomial_bounds}

It is easy to show that \( \lan_k(p) = p(\vec{A})\vec{b} \) for any polynomial \( p \) with \( \deg p < k \); see for example \cite{druskin_knizhnerman_89,saad_92}.
This implies that \( \lan_k(f) = p_{k}(\vec{A})\vec{b} \), where \( p_{k} \) is the degree \(k-1\) polynomial interpolating \( f \) at the eigenvalues of \( \vec{T}_k \). Since eigenvalues of \( \vec{A} \) are often approximated by eigenvalues of \( \vec{T}_k \), this interpolating polynomial is a sensible approximation.

More formally, let \( \| \cdot \| \) be any norm induced by a positive definite matrix which commutes with $\vec{A}$; i.e. with the same eigenvectors as $\vec{A}$. Such norms include the 2-norm, the $\vec{A}^2$-norm, and the $\vec{A}$-norm (if $\vec{A}$ is positive definite).
Then \( \| g(\vec{A}) \vec{v} \| \leq \| g(\vec{A}) \|_2 \cdot \| \vec{v} \| \) for any \( g : \mathbb{R}\to\mathbb{R} \), so by the triangle inequality, for any \( p \) with \( \deg p < k \),
\begin{align*}
\| f(\vec{A}) \vec{b} - \lan_k(f) \|
& \leq \| f( \vec{A} ) \vec{b} - p( \vec{A} ) \vec{b} \| + \| p( \vec{A} ) \vec{b} - \lan_k (p) \| + \| \lan_k (p) - \lan_k (f) \| 
\\& = \| (f( \vec{A} ) - p( \vec{A} ) )\vec{b} \| + 0 + \| \vec{Q}_k ( p( \vec{T}_k ) - f( \vec{T}_k ) ) \vec{Q}_k^{\cT} \vec{b} \| 
\\& \le \| f( \vec{A} ) - p( \vec{A} ) \|_2 \cdot \| \vec{b} \| + \| \vec{Q}_k ( p( \vec{T}_k ) - f( \vec{T}_k ) ) \vec{Q}_k^{\cT} \|_2 \cdot \|\vec{b} \|
\\& \le \left( \| f( \vec{A} ) - p( \vec{A} ) \|_2 + \| p(\vec{T}_k ) - f( \vec{T}_k ) \|_2 \right) \cdot \| \vec{b} \| . 
\end{align*}
Denote the infinity norm of a scalar function \( h:\R\to\R \) over \( S\subset \R \) by \( \|h\|_S  := \sup_{x\in S}| h(x)| \).
Then, writing the set of eigenvalues of a Hermitian matrix \( \vec{B} \) as \( \Lambda(\vec{B}) \),
\begin{align}
\| f(\vec{A}) \vec{b} - \lan_k(f) \|
    &\leq \min_{\deg p< k} \left( \| f-p \|_{\Lambda(\vec{A})} + \| f - p \|_{\Lambda(\vec{T}_k)} \right) \| \vec{b} \|. \label[ineq]{eqn:triangle_ineq1}
\end{align}
Finally, introducing the notation \(\mathcal{I}(\vec{B}) := [\lambda_{\text{min}}(\vec{B}), \lambda_{\text{max}}(\vec{B})]\) and using the fact that \( \Lambda(\vec{T}_k) \subset \mathcal{I}(\vec{A}) \), we obtain the classic bound
\begin{align}
\label[ineq]{eqn:poly_unif}
\| f(\vec{A}) \vec{b} - \lan_k(f) \|_2 
    %&\leq  2 \left( \min_{\deg p<k} \max_{x \in \mathcal{I}(\vec{A})} | f(x) - p(x) | \right) \| \vec{b} \|_2.
    &\leq  2 \min_{\deg p<k} \left( \| f-p \|_{\mathcal{I}(\vec{A})}  \right) \| \vec{b} \|_2.
\end{align}
That is, except for a possible factor of \( 2 \), the error of the Lanczos-FA approximation to \( f( \vec{A} ) \vec{b} \) is at least as good as the best \emph{uniform polynomial approximation to \( f \)} on the interval containing the eigenvalues of \( \vec{A} \).
For arbitrary \( f \), \cref{eqn:poly_unif} remains the standard bound for Lanczos-FA. 
It   has been studied carefully and is known to hold to a close degree in finite precision arithmetic \cite{musco_musco_sidford_18}.

However, the uniform error bound of \cref{eqn:poly_unif} is often too loose to accurately predict the performance of Lanczos-FA.
Notably, it depends only on the range of eigenvalues \(\mathcal{I}(\vec{A})\) and not on more fine-grained information like the presence of eigenvalue clusters or isolated eigenvalues, which are known to lead to faster convergence.
The expression in \cref{eqn:triangle_ineq1} is more accurate but it cannot be used as an a priori bound since it involves the eigenvalues of the tridiagonal matrix $\vec T_k$, which depend on $\vec b$. It also cannot be used as a practical a posteriori bound since it involves all  eigenvalues of \( \vec{A} \).

The goal of this paper is to address these limitations. Before doing so, we discuss an example to better illustrate why \cref{eqn:poly_unif} can be  loose as an a priori bound.
It is well known that the eigenvalues of \( \vec{T}_k \) are interlaced by those of \( \vec{A} \); that is, \( \Lambda(\vec{T}_k) \subset \mathcal{I}(\vec{A}) \) and between each pair of eigenvalues of \( \vec{T}_k \) is at least one eigenvalue of \( \vec{A} \).
With this property in mind, define \( \mathcal{J}_k(\vec{A}) \) as the set of all \( k \)-tuples \( \bm{\mu} = ( \mu_1 , \ldots , \mu_k ) \in \mathbb{R}^k \) that are interlaced by the eigenvalues of \( \vec{A} \). %; that is, \( \lambda_{\text{min}} ( \vec{A} ) < \min_j \mu_j \), \( \lambda_{\text{max}} ( \vec{A} ) > \max_j \mu_j \), and between each consecutive pair \( \mu_i , \mu_j \) there is at least one eigenvalue of \( \vec{A} \).
Then, we can use \cref{eqn:triangle_ineq1} to write
\begin{align}
\| f(\vec{A}) \vec{b} - \lan_k(f) \| 
% & \leq  \max_{\bm{\mu}\in \mathcal{J}_k(\vec{A})} \min_{\deg p<k} \left( \max_{\lambda \in \Lambda(\vec{A})} | f(\lambda) - p(\lambda) | +\max_{j=1, \ldots k} | f(\mu_j) - p(\mu_j) | \right) \| \vec{b} \|. \label[ineq]{eqn:triangle_ineq2} 
 & \leq  \max_{\bm{\mu}\in \mathcal{J}_k(\vec{A})} \min_{\deg p<k} \left( \| f-p \|_{\Lambda(\vec{A})} + \| f - p \|_{\bm{\mu}} \right) \| \vec{b} \|. \label[ineq]{eqn:triangle_ineq2} 
\end{align}
The bound \cref{eqn:triangle_ineq2} is an a priori error bound, and at least in some special cases, provides more insight than \cref{eqn:poly_unif} in situations where the eigenvalues of $\vec{A}$ are clustered.

\begin{example}
\label{ex:unif_discrete}
Consider \( \vec{A} \) with many eigenvalues uniformly spaced through the interval \( [0,1] \) and a single isolated eigenvalue at \( \kappa > 1 \).
Since the eigenvalues of \( \vec{T}_k \) are interlaced by those of \( \vec A \), there is at most one eigenvalue of \( \vec{T}_k \) between \( 1 \) and \( \kappa \); that is, \( \Lambda( \vec{A} )\cup\Lambda( \vec{T}_k ) \) is contained in \( [0,1] \cup \{ \mu , \kappa \} \) for some \( \mu\in[1,\kappa] \).
We then have
\begin{align}
\label[ineq]{eqn:ex_minimax_bd}
    \| f(\vec{A}) \vec{b} - \lan_k(f) \| 
%\leq \max_{\mu\in[1,\kappa]} \min_{\deg p<k} \left( \max_{\lambda\in[0,1]\cup\{\mu,\kappa\}} | f(\lambda)-p(\lambda)| \right) \| \vec{b} \|.
\leq 2 \max_{\mu\in[1,\kappa]} \min_{\deg p<k} \left( \| f-p \|_{[0,1]\cup\{\mu,\kappa\}} \right) \| \vec{b} \|.
\end{align}
For \( \kappa = 5\), \( f(x) = \exp(-x) \), and \( k=6 \), we use a numerical optimizer to determine that the value  maximizing the right hand side of \cref{eqn:ex_minimax_bd} is  \( \mu^*\approx 4.96\).
In \cref{fig:lanc_poly_approx} we show the error of the Lanczos-FA polynomial along with the optimal uniform polynomial approximations to \( f \) on \( [0,5] \), which contains \( [0,1]\cup\{\mu^*,5\} \).
Here the optimal uniform polynomial approximation is computed by the Remez algorithm.
As expected, the bound from \cref{eqn:ex_minimax_bd}  is significantly better than that from the uniform approximation.

\begin{figure}[ht]
    \includegraphics[width=.97\textwidth]{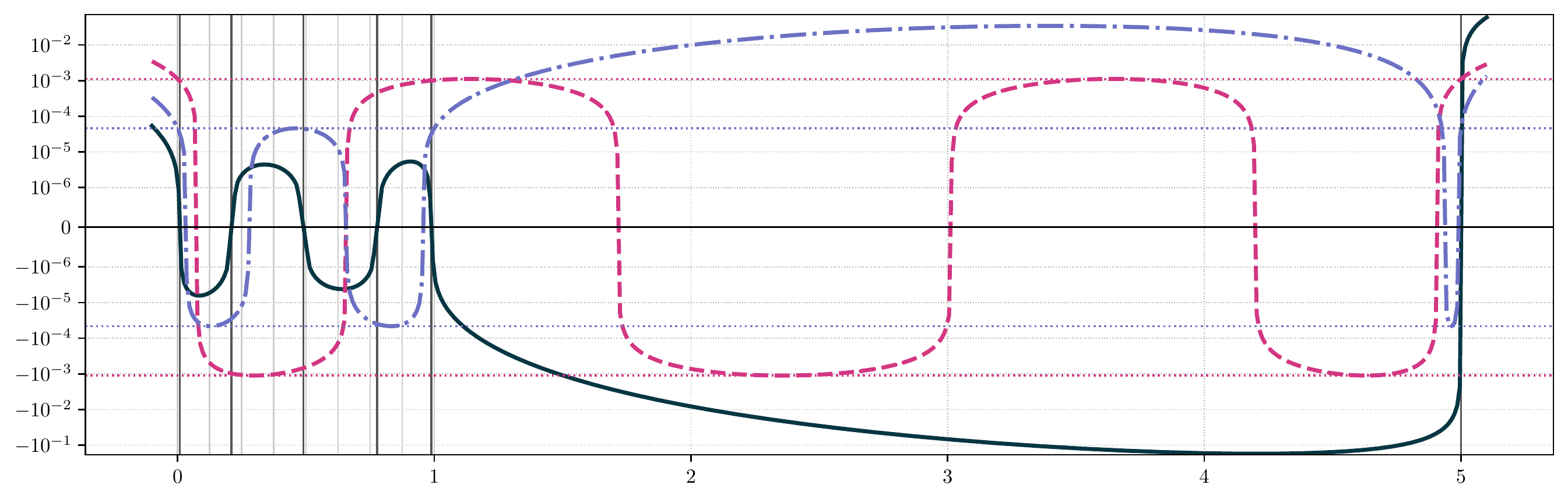}

   \caption{
    Comparison of errors of degree 5 polynomial approximations to \( f(x) = \exp(-x) \).
    \emph{Legend}:
    Lanczos-FA approximation for $\vec b$ with equal projection onto all eigenvectors of $\vec A$ ({\protect\raisebox{0mm}{\protect\includegraphics[scale=.7]{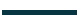}}}), 
    optimal uniform approximation on \( [0,5] \) ({\protect\raisebox{0mm}{\protect\includegraphics[scale=.7]{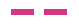}}}), 
    optimal uniform approximation on \( [0,1]\cup\{\mu^*,5\} \) ({\protect\raisebox{0mm}{\protect\includegraphics[scale=.7]{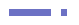}}}).
    The light vertical lines are the eigenvalues of \( \vec{A} \), while the darker vertical lines are the eigenvalues of \( \vec T_6 \) (the Ritz values).
    \emph{Remarks}: 
    Note that the Lanczos-FA approximation becomes very inaccurate on \( (1,5) \) which allows a smaller error on the eigenvalues of $\vec A$, which is the only error that impacts our approximation to $f(\vec A)\vec b$.
    As a result, the uniform approximation on \( [0,1] \cup \{\mu^*,5\} \) is a much better bound for the Lanczos-FA error than the uniform approximation on \( [0,5] \), which remains equally accurate over the entire interval \([0,5]\). 
    }
    \label{fig:lanc_poly_approx}
\end{figure}
\end{example}

\subsection{Our Approach and Roadmap}

Given the potential looseness of the classic uniform error bound on Lanczos-FA \cref{eqn:poly_unif}, our goal is to derive tighter, but still practically computable  error bounds. Ideally, we want bounds that are both generally applicable and easier to apply than e.g., the bound of \cref{eqn:triangle_ineq2} based on interlacing.

One important case where such bounds already exist is when $f(x)=1/x$ and $\vec{A}$ is positive definite.
In this setting, tight a posteriori error bounds are easily obtained by computing the residual \( \| \vec{A}\lan_k(f)- \vec{b} \| \), and moreover, much stronger a priori error bounds are known than \cref{eqn:poly_unif}. In particular, \( \| f(\vec{A})\vec{b} - \lan_k(f) \| \) is equal to the error of the conjugate gradient algorithm (CG) used to solve \( \vec{A} \vec{x} = \vec{b} \) and therefore optimal over the Krylov subspace in the \( \vec{A} \)-norm. This immediately implies a priori bounds depending only on $\min_{\deg p< k} \| f-p \|_{\Lambda(\vec{A})}$, and so can be much tighter than \cref{eqn:poly_unif} for matrices with clustered or isolated eigenvalues (see \cref{sec:linear_systems} for details).

Our approach is inspired by these sharper a posteriori and a priori error bounds for Lanczos-FA in the case of linear systems -- i.e., for $f(x)=1/x$. We exploit the existence of these bounds to address a more general class of functions 
%We do not discuss in detail how to bound this linear system error  -- there are many known approaches, both a priori and a posteriori, and the best bounds to use are often context dependent. 
%For a more detailed discussion we refer readers to \cref{sec:linear_systems}.
by using the Cauchy integral formula to write the Lanczos-FA error \( f(\vec{A})\vec{b} - \lan_k(f) \) for any analytic $f$ in terms of the Lanczos error for solving a continuum of shifted linear systems in $\vec A$. We then bound this error in terms of the error in computing the solution to  a \emph{single shifted system}, $(\vec A-w \vec I)^{-1} \vec b$. This reduction is presented in \cref{sec:cif_bound}, along with a discussion of related work. We proceed, in \cref{sec:results}, to show how this reduction can be used to obtain useful a priori and a posteriori error bounds. One highlight result is a proof that, for any analytic function $f$, the relative error of Lanczos-FA in approximating \(f(\vec A)\vec b\) can be bounded by a fixed constant times the relative error in solving a slightly shifted linear system in \(\vec A\). % \Chris{Not sure best way to say the previous sentence correctly and concisely.}
We provide examples and numerical experiments that illustrate the quality of our bounds in  \cref{sec:examples}. In
\cref{sec:finite_precision} we give an analysis of our bounds in finite precision. Finally, in \cref{sec:quadratic_form}, we discuss generalizations  to quadratic forms \( \vec{b}^\cT f(\vec{A})\vec{b} \).

\section{Lanczos-FA error and the Cauchy integral formula}
\label{sec:cif_bound}

Assuming \( f : \mathbb{C} \rightarrow \mathbb{C} \) is analytic in a neighborhood of the eigenvalues of \( \vec{A} \) and \( \Gamma \) is a simple closed curve or union of simple closed curves inside that neighborhood and enclosing the eigenvalues of \( \vec{A} \), the Cauchy integral formula states that
\begin{align}
    \label{eqn:fofACIF1}
    f(\vec{A})\vec{b} = - \frac{1}{2 \pi i} \oint_{\Gamma} f(z) (\vec{A} - z \vec{I} )^{-1} \vec{b} \, \d{z}.
\end{align}
If \( \Gamma \) also encloses the eigenvalues of \( \vec{T}_k \) we can similarly write the Lanczos-FA approximation as
\begin{align}
\label{eqn:fofT_CIF}
    \vec{Q}_k f(\vec{T}_k) \vec{Q}_k^\cT \vec{b}
    = - \frac{1}{2 \pi i} \oint_{\Gamma} f(z) \vec{Q}_k (\vec{T}_k - z \vec{I} )^{-1} \vec{Q}_k^\cT \vec{b} \, \d{z} .
\end{align}
Observing that the integrand of \cref{eqn:fofACIF1} contains the solution to the shifted linear system \( (\vec{A}-z\vec{I}) \vec{x} = \vec{b} \) while  \cref{eqn:fofT_CIF} contains the Lanczos-FA approximation to the solution, we make the following definition.
\begin{definition}\label{def:err}
For \( z \in \mathbb{C} \), %let \( h_z (x) = 1/(x-z) \).
define the \( k \)-th Lanczos-FA error and residual for the linear system \( (\vec{A}-z\vec{I}) \vec{x} = \vec{b} \) as,
\begin{align*}
    \err_k(z,\vec{A},\vec{b}) &:= (\vec{A} - z \vec{I})^{-1} \vec{b} - \vec{Q}_k(\vec{T}_k-z\vec{I})^{-1}\vec{Q}_k^\cT \vec{b}%\lan_k ( h_z )
    ,\\
    \Res_k(z,\vec{A},\vec{b}) &:= \vec{b} - (\vec{A} - z \vec{I}) \vec{Q}_k(\vec{T}_k-z\vec{I})^{-1}\vec{Q}_k^\cT \vec{b}.%\,\lan_k ( h_z ).
\end{align*}
As with the Lanczos-FA approximation, we will typically omit the arguments \( \vec{A} \) and \( \vec{b} \), and in the case \( z = 0 \), we will often write \( \err_k \) and \( \Res_k \).
\end{definition}
%\Tyler{Should we mention CG at some point?} \Anne{No need.}
With \cref{def:err} in place, the error of the Lanczos-FA approximation to \( f(\vec{A})\vec{b} \) can be written as
\begin{align}
    \label{eqn:lanczos_FA_error}
    f(\vec{A}) \vec{b} - \vec{Q}_k f( \vec{T}_k ) \vec{Q}_k^{\cT} \vec{b}
    &= - \frac{1}{2 \pi i} \oint_{\Gamma} f(z) \, \err_k(z) \, \d{z}.
\end{align}
Therefore, if for every $z\in\Gamma$ we are able to understand the convergence of Lanczos-FA on the linear system \( (\vec{A}-z\vec{I}) \vec{x} = \vec{b} \), then this formula lets us understand the convergence of Lanczos-FA for \( f(\vec{A})\vec{b} \).
To simplify bounding \cref{eqn:lanczos_FA_error}, we will write $\err_k(z)$ for all  \( z\in\Gamma \) in terms of the error in solving a single shifted linear system.

%We now turn our attention to the shifted linear system \( (\vec{A}-z \vec{I}) \vec{y} = \vec{b} \), where \( z \) is arbitrary and possibly complex.
To do this, we use the fact that the Lanczos factorization \cref{eqn:lanczos_factorization} can be shifted, even for complex \( z \), to obtain
\begin{align}
    \label{eqn:shifted_lanczos_factorization}
    ( \vec{A} - z \vec{I} ) \vec{Q}_k
    &= \vec{Q}_k ( \vec{T}_k - z \vec{I} ) + \beta_k \vec{q}_{k+1} \vec{e}_k^{\T}.
\end{align}
That is, Lanczos applied to \( (\vec{A},\vec{b}) \) for \( k \) steps produces output \( \vec{Q}_k \) and \( \vec{T}_k \) satisfying \cref{eqn:lanczos_factorization} while Lanczos applied to \( (\vec{A} - z \vec{I}, \vec{b}) \) for \( k \) steps produces output \( \vec{Q}_k \) and \( \vec{T}_k-z \vec I \) satisfying \cref{eqn:shifted_lanczos_factorization}. %, where \( \vec{Q}_k \) and \( \vec{T}_k \) are the same in both factorizations. \Cam{Are these the ones produced by iterating on $\vec A$ right? We should say explicitly.}
%In the case that \( z \) is not real, we interpret the inner products in the Lanczos method as complex inner products linear in the first argument; i.e. \( \langle a \vec{x} , \vec{y} \rangle = a \langle \vec{x}, \vec{y} \rangle \) for all \( a \in \mathbb{C} \).
Using this fact, we have the following well known lemma.
\begin{lemma}
\label{thm:shifted_lanczos_equivalence}
For all \( z \) where \( \vec{T}_k - z \vec{I} \) is invertible, 
\begin{align*}
    \Res_k(z) 
%    = \beta_k \vec{e}_k^{\T} (\vec{T}_k - b\vec{I})^{-1} \vec{e}_1 \vec{q}_{k+1}
    = \left( \frac{(-1)^{k}}{\det(\vec{T}_k -z \vec{I}) }\prod_{j=1}^{k} \beta_j \right) \| \vec{b} \|_2\: \vec{q}_{k+1}.
\end{align*}
\end{lemma}
%\Cam{Can probably eventually appendiczie this proof if well known.} \Anne{Since it is short, I prefer to keep it here.}
%\Tyler{Well known is maybe relative. People who are particularly familiar with the analysis of CG will know this, but maybe not general numerical analysis/scientific computing people}
\begin{proof}
    From \cref{eqn:shifted_lanczos_factorization}, and the fact that $\vec {Q}_k$'s first column is $\vec b/\|{\vec b}\|_2$, it is clear that,
    \begin{align*}
        (\vec{A} - z\vec{I}) \vec{Q}_k(\vec{T}_k-z\vec{I})^{-1}\vec{Q}_k^\cT \vec{b}
        &= (\vec{A}-z \vec{I}) \vec{Q}_k (\vec{T}_k - z \vec{I})^{-1} \| \vec{b} \|_2 \vec{e}_1
        \\&= \vec{Q}_k \| \vec{b} \|_2 \vec{e}_1 + \beta_k \vec{q}_{k+1} \vec{e}_k^{\T} (\vec{T}_k - z \vec{I})^{-1} \| \vec{b} \|_2 \vec{e}_1
        \\&= \vec{b} + \beta_k \vec{q}_{k+1} \vec{e}_k^{\T} (\vec{T}_k - z \vec{I})^{-1} \| \vec{b} \|_2 \vec{e}_1.
    \end{align*}

    Using the formula \( (\vec{T}_k - z \vec{I} )^{-1} = (1/ \det(\vec{T}_k - z \vec{I})) \mbox{ adj}( \vec{T}_k - z \vec{I} ) \), we see that
    \begin{align*}
        \vec{e}_k^{\T} (\vec{T}_k - z \vec{I})^{-1} \vec{e}_1
        = \frac{(-1)^{k-1}}{\det( \vec{T}_k - z \vec{I} )}\prod_{j=1}^{k-1} \beta_j.
        \tag*{\qed}
    \end{align*}
\end{proof}
We use \cref{thm:shifted_lanczos_equivalence} to relate \( \err_k(z) \) to \( \err_k(w) \) for any $z,w \in \mathbb{C}$.

\begin{definition}
For $w,z\in\mathbb{C}$ define $h_{w,z}:\R\to\mathbb{C}$ and $h_z:\R\to\mathbb{C}$ by
\begin{align*}
    h_{w,z}(x) := \frac{x-w}{x-z}
    ,&&
    h_z(x):= \frac{1}{x-z}
\end{align*}
\end{definition}

\begin{corollary} 
For all \( z , w \in \mathbb{C} \), where \( \vec{A} - z \vec{I} \) and \( \vec{A} - w \vec{I} \) are both invertible,
%\( \vec{A} - b \vec{I} \) is invertible. Then,
\label{thm:shifted_linear_system_error}
\begin{align*}
    \err_k(z)
    &= \det(h_{w,z}(\vec{T}_k))  h_{w,z}(\vec{A}) 
    \,\err_k(w)
    \\
    \Res_k(z)
    &= \det(h_{w,z}(\vec{T}_k))
    \,\Res_k(w).
\end{align*}
\end{corollary}

\begin{proof}
By \cref{thm:shifted_lanczos_equivalence},
\begin{align*}
    \det(\vec{T}_k - z\vec{I}) \,\Res_k (z) = \det(\vec{T}_k - w \vec{I}) \,\Res_k (w). 
\end{align*}
Thus,
\begin{align*}
    \Res_k (z) = \frac{\det(\vec{T}_k - w\vec{I})}{\det(\vec{T}_k - z\vec{I})} \,\Res_k (w) = \det(h_{w,z}(\vec{T}_k)) \,\Res_k (w) .
%    \frac{\det(\vec{T}_k-w\vec{I})}{\det(\vec{T}_k - z\vec{I})} (\vec{A} - w\vec{I}) \err_k (w)
%    = (\vec{A} - z \vec{I}) \err_k(z)
\end{align*}
Noting that \( \Res_k (z) = (\vec{A}-z\vec{I}) \,\err_k(z) \) and 
\( \Res_k (w) = (\vec{A}-w\vec{I}) \err_k(w) \), we obtain the relation between the
errors,
%Thus,
\begin{align*}
    \err_k(z)
    &= \det(h_{w,z}(\vec{T}_k)) (\vec{A} - z \vec{I})^{-1} ( \vec{A} - w \vec{I})\, \err_k ( w ) 
    \\&= \det(h_{w,z}(\vec{T}_k))  h_{w,z}(\vec{A}) \,\err_k(w) .
    \tag*{\qed}
\end{align*}
\end{proof}

In summary, combining \cref{eqn:lanczos_FA_error} and \cref{thm:shifted_linear_system_error} we have the following corollary.
This result is by no means new, and appears throughout the literature; see for instance \cite{frommer_simoncini_09} and \cite[Theorem 3.4]{frommer_guttel_schweitzer_14}.
\begin{corollary}\label{thm:integral_error_vec} 
Suppose \( \vec{A} \) is a Hermitian matrix and \( f : \mathbb{C} \rightarrow \mathbb{C} \) is a function analytic in a neighborhood of the eigenvalues of \( \vec{A} \) and \( \vec{T}_k \), where  \( \vec{T}_k \) is the tridiagonal matrix output by Lanczos run on \( \vec{A},\vec{b} \) for \( k \) steps.
Then, if \( \Gamma \) is a simple closed curve or union of simple closed curves inside this neighborhood and enclosing the eigenvalues of \( \vec{A} \) and \( \vec{T}_k \) and \( w\in\mathbb{C} \) is such that \( w \not\in \Lambda(\vec{T}_k)\cup\Lambda(\vec{A}) \),
\begin{align*}
    f(\vec{A}) \vec{b} - \lan_k(f)
    &= \left( - \frac{1}{2\pi i} \oint_{\Gamma} f(z) \det(h_{w,z}(\vec{T}_k)) h_{w,z}(\vec{A}) \d{z} \right) \, \err_k(w).
\end{align*}
\end{corollary}
%\Cam{Should we make the above equation a corrloary?}

\subsection{Bound on Lanczos-FA error in terms of linear system error}

Our main result is a flexible bound for the  Lanczos-FA error, obtained by bounding the integral in the right-hand side of \cref{thm:integral_error_vec}. As we will see in \cref{sec:results}, we can instantiate this theorem to obtain effective a priori and a posteriori error bounds in many settings.
% the Cauchy integral formula \cref{eqn:lanczos_FA_error}, and bounding each $\err_k(z)$ term with $\err_k(w)$ for a single choice of $w$.
%Given a scalar function \( h:\R\to\R \) and set \( S\subset \R \), we denote the infinity norm of \( h \) on \( S \) by \( \|h\|_S \); i.e. \( \|h\|_S := \sup_{x\in S}| h(x)| \).
%Using this notation, we then have the following theorem.
\begin{theorem}
\label{thm:err_int}
Consider the setting of \cref{thm:integral_error_vec}. If, additionally, for some $S_0,S_1,\ldots,S_k \subset \mathbb{R}$ we have \( \Lambda(\vec{A})\subset S_0\ \) and \( \lambda_i(\vec{T}_k) \in S_i \) for  \( i=1,\ldots, k\), then
\begin{align*}
    \| f(\vec{A})\vec{b} - \lan_k(f) \|
    %\leq \underbrace{\vphantom{ \bigg| }\left( \frac{1}{2\pi}\oint_{\Gamma} |f(z)| \cdot \left(\prod_{i=1}^{k} \|h_{w,z}\|_{S_i}\right) \cdot \|h_{w,z}\|_{S_0} \cdot | \d{z} | \right)}_{\text{integral term}}  \hspace{-1.2 em}\underbrace{\vphantom{ \Bigg| } \| \err_k(w) \| . \hspace{-.4em} }_{\text{linear system error}} \hspace{-.5em}
    \leq \underbrace{\vphantom{ \bigg| }\left( \frac{1}{2\pi}\oint_{\Gamma} |f(z)| \cdot \left(\prod_{i=1}^{k} \| h_{w,z}\|_{S_i}\right) \cdot \|h_{w,z}\|_{S_0} \! \cdot | \d{z} | \right)}_{\text{integral term}}  \hspace{-1.2 em}\underbrace{\vphantom{ \Bigg| } \| \err_k(w) \| . \hspace{-.4em} }_{\text{linear system error}} \hspace{-.5em}
\end{align*}
\end{theorem}
The above  bound depends on our choices of $\Gamma$, $w$, and the sets $S_0, S_1,\ldots, S_k$, which must contain the eigenvalues of $\vec A$ and $\vec T_k$. 
The sets $S_0, S_1,\ldots, S_k$ should be chosen based on the informatoin we have about $\vec A$ and $\vec T_k$. 
For example, we could take all these sets to be the eigenvalue range $\mathcal{I}( \vec A)$. 
If we have more information a priori about the eigenvalues of $\vec A$, we can obtain a tighter bound by choosing smaller $S_0$, with correspondingly lower $\|h_{w,z}\|_{S_0}$. 
For an a posteriori bound, we can simply set $S_i = \{ \lambda_i(\vec T_k) \}$, for $i = 1,\ldots,k$. This gives an optimal value for $\|h_{w,z}\|_{S_i}$. Both approaches are detailed in \cref{sec:results}.

We emphasize that the integral term and linear system error term in the theorem are entirely decoupled.
Thus, once the integral term is computed, bounding the error of Lanczos-FA for \( f(\vec{A})\vec{b} \) is reduced to bounding \( \| \err_k(w) \| \), and if the integral term can be bounded independently of \( k \), \cref{thm:err_int} implies that, up to a constant factor, the Lanczos-FA approximation to \( f(\vec{A})\vec{b} \) converges at least as fast as \( \| \err_k(w) \| \).

\begin{proof}[Proof of \cref{thm:err_int}]
Applying the triangle inequality for integrals and the submultiplicativity of matrix norms to \cref{thm:integral_error_vec} we have
\begin{align}
    \label[ineq]{eqn:integral_error}
    \| f(\vec{A}) \vec{b} - \lan_k(f) \|
    \leq \left( \frac{1}{2\pi}\oint_{\Gamma} |f(z)| \cdot |\det(h_{w,z}(\vec{T}_k))| \cdot \| h_{w,z}(\vec{A}) \|_2 \!\cdot |\d{z} | \right) \| \err_k(w) \| .
\end{align}
Next, since %as an immediate consequence of the definition of \( \| h_{w,z} (\vec{A}) \|_2 \), if 
\( \Lambda(\vec{A}) \subseteq S_0 \) then
\begin{align*}
%    \label[ineq]{eqn:hwz_qwz}
    \| h_{w,z}(\vec{A}) \|_2 = \max_{i=1,\ldots, n} | h_{w,z}(\lambda_i(\vec{A})) | \leq \|h_{w,z}\|_{S_0},
\end{align*}
and similarly, if  \( \lambda_i(\vec{T}_k) \in S_i \) for \( i=1, \ldots , k, \) then 
\begin{align}
    \label[ineq]{eqn:dkwz_qwz}
    |\!\det(h_{w,z}(\vec{T}_k))| = \left| \prod_{i=1}^{k} h_{w,z}(\lambda_i(\vec{T}_k)) \right| \leq
    \prod_{i=1}^{k} \|h_{w,z}\|_{S_i}.
\end{align}
Combining these inequalities yields the result.
\end{proof}
%root-condition number bound on favorable spectra. 

\subsection{Comparison with previous work}
%\Tyler{Is this okay here? If we should move it back to after section 2, I think it should be its own section and perhaps beefed up a bit.}

%\Cam{I think before this section we need to at least say what our framework is, outside the abstract. Should we put this at the end of section 2? Otherwise we keep comparing things prior work does in comparision to what we do, but we haven't talked about our approach at all yet.}

Our framework for analyzing Lanczos-FA has four properties which differentiate it from past work:
(i) it is applicable to a wide range of functions, 
(ii) it yields a priori bounds dependent on fine-grained properties of the spectrum of \( \vec{A} \) such as clustered or isolated eigenvalues, 
(iii) it can be used a posteriori as a practical stopping criterion,
and (iv) it is applicable when computations are carried out in finite precision arithmetic. %when used with a perturbed Lanczos recurrence.
To the best of our knowledge, no existing analysis satisfies more than two of these properties simultaneously.
In this section, we provide a brief overview of the most relevant past work.

Most directly related to our framework is a series of works which also make use of the shift-invariance of Krylov subspaces when \( f \) is a Stieltjes function\footnote{A function \( f \) defined on the positive real axis is a Stieltjes function if and only if \( f(x) \geq 0 \) for all $x\in\R$ and \( f \) has an analytic extension to the cut plane \( \mathbb{C} \setminus (-\infty,0] \) satisfying \( \Im(f(x)) \leq 0 \) for all \( x \) in the upper half plane \cite[Theorem 3.2]{berg_07} \cite[p. 127 attributed to Krein]{aheizer_65}.} \cite{frommer_guttel_schweitzer_14a,frommer_schweitzer_15,ilic_turner_simpson_09} or a certain type of rational function \cite{frommer_kahl_lippert_rittich_13,frommer_simoncini_08b,frommer_simoncini_09}.
These analyses are applicable a priori and a posteriori and in fact allow for corresponding error \emph{lower bounds} as well.
However, these bounds cannot be applied to more general functions, and the impact of a perturbed Lanczos recurrence in finite precision is not considered.

The most detailed generally applicable analysis is \cite{musco_musco_sidford_18}, which extends \cite{druskin_knizhnerman_91,druskin_knizhnerman_95} and studies \cref{eqn:poly_unif} when Lanczos is run in finite precision.
However, as discussed in \cref{sec:polynomial_bounds}, \cref{eqn:poly_unif} is often too pessimistic in practice as it does not depend on the fine-grained properties about the distribution of eigenvalues.
%\Anne{I wonder if we should explain somewhere what we mean by a priori and a posteriori.}
%\Tyler{There is a breif comment at the beginning of 3.1, but I think this coudld be a good idea}
Another generally applicable analysis is \cite{hochbruck_lubich_selhofer_98}, which suggests replacing \( \err_k(z) \) with \( \Res_k(z) \) in \cref{eqn:lanczos_FA_error}.
Since \( \Res_k(z) \) can be computed once the outputs of Lanczos have been obtained, the resulting integral can be computed (or at least approximated by a quadrature rule).
However, this approach does not take into account the actual relationship between \( \Res_k(z) \) and \( \err_k(z) \), and therefore gives only an estimate of the error, not a true bound.
Another Cauchy integral formula based approach is \cite{hochbruck_lubich_97} which shows that Lanczos-FA exhibits superlinear convergence for the matrix exponential and certain other specific analytic functions.

There are a variety of other bounds specialized to individual functions.
For example, it is known that if \( \vec{A} \) is nonnegative definite and \( t > 0 \), then the error in the Lanczos-FA approximation for the matrix exponential \( \exp( t \vec{A}) \vec{b} \) can be related to the maximum over \( s \in [0,t] \) of the error in the  optimal approximation to \( \exp( s \vec{A}) \vec{b} \) over a Krylov space of slightly lower dimension \cite{druskin_greenbaum_knizhnerman_98}.
More recent work involving the matrix exponential are \cite{jia_lv_14,jawecki_auzinger_koch_19,jawecki_21}.
%in exact and finite precision arithmetic \cite{druskin_greenbaum_knizhnerman_98}.
There is also a range of work which analyzes the convergence of Lanczos-FA and related methods for computing the square root and sign functions \cite{boricci_99,borici_03,eshof_frommer_lippert_schilling_van_der_vorst_02}.
%Further results for the matrix exponential \cite{orecchia_sachdeva_nisheeth_12}.

\section{Applying our framework}
\label{sec:results}

We proceed to show how to effectively bound the integral term of \cref{thm:err_int}, to give  a priori and a posteriori  bounds on the Lanczos-FA error, assuming accurate bounds on \( \| \err_k(w) \| \) are available.
Throughout, we assume $w\in\mathbb{R}$ and we do not discuss in detail how to bound this linear system error  -- there are many known approaches, both a priori and a posteriori, and the best bounds to use are often context dependent. 
For a more detailed discussion we refer readers to \cref{sec:linear_systems}.

To use \cref{thm:err_int}, we must evaluate or bound \( \|h_{w,z}\|_{S_i} \).
Towards this end, we introduce the following lemmas, which apply when  $S_i$ is an interval.
These lemmas are also useful when $S_i$ is a union of intervals -- in that case $\|h_{w,z}\|_{S_i}$ is bounded by the maximum bound on any of these intervals.
i
\begin{lemma}
\label{thm:Q_wz_value}
    For any interval \( [a,b] \subset \mathbb{R} \), if \( z \in \mathbb{C} \setminus [a,b] \) and \( w\in\mathbb{R} \), we have
\begin{align*}
    \|h_{w,z}\|_{[a,b]}
    %Q_{w,z} = \max_{x \in \mathcal{I}(\vec{A}) } \left| \frac{x-w}{x-z} \right| 
    = \max \left\{ 
    \left| \frac{a-w}{a-z} \right|, 
    \left| \frac{b-w}{b-z} \right|,
    \left( \left| \frac{z-w}{\Im(z)} \right|~\text{ if } x^{*} \in [a,b] ~\text{else}~0 \right)
    \right\}
\end{align*}
where 
\begin{align*}
    x^* := \frac{\Re(z)^2 + \Im(z)^2 - \Re(z) w}{\Re(z)-w}.
\end{align*}

\end{lemma}
\begin{proof}
    Note that for $x\in\mathbb{R}$, 
\begin{align*}
    | h_{w,z} (x) |^2 = \left| \frac{x-w}{x-z} \right|^2 = \frac{(x-w )^2}{( x - \Re(z) )^2 + \Im(z)^2} ,
\end{align*}
and
\begin{align*}
    \frac{\d}{\d{x}} \left( | h_{w,z} (x) |^2 \right) = \frac{[ (x - \Re(z) )^2 + \Im(z )^2 ]2(x-w) - (x-w )^2 2 ( x - \Re(z) )}{[ (x - \Re(z) )^2 + \Im(z )^2 ]^2} .
\end{align*}
%Define 
%\begin{align*}
%    g(x) = \frac{x-w}{\sqrt{(\Re(z)-x)^2 + \Im(z)^2}}.
%\end{align*}
%    g'(x) = \frac{\Im(z)^2 + (\Re(z)-w)(\Re(z)-x)}{(\Re(z)-x)^2 + \Im(z)^2}.
Aside from \( x=w \), where \( h_{w,z} (x) = 0 \), the only value \( x \in \R \) for which \( \frac{\d}{\d{x}} \left( | h_{w,z} (x) |^2 \right) = 0 \) is \( x^* \).
%is the only value \( x\in \R \) for which \( g'(x) = 0 \).
This implies that the only possible local extrema of  \( | h_{w,z} (x)|  \) on \( [a,b] \) are  \( a \), \( b \), and \( x^* \) if \( x^* \in [a,b] \).
Substituting the expression for \( x^{*} \) into that for \( | h_{w,z} ( x^{*} ) | \), one finds, after some algebra, that \( | h_{w,z} ( x^{*} ) | = | z-w | / | \Im (z) | \).
%We have that \( | h_{w,z}(x) | = |g(x)| \)
%Moreover, \( g(x) \) only changes sign at \( x=w \) and \( g(w) = 0 \), 
%so the only possible local maxima of \( | h_{w,z}(x) | \) on \( [a,b] \) are also  \( %a \), \( b \), and \( x^* \) if \( x^* \in [a,b] \).
%The result follows from the fact that \( | g(x^*) | = |z-w|/|\Im(z)| \).
\iffalse
\Tyler{the last sentence could be replaced by the full derivation:
It remains to show that \( | h_{w,z} (x^*) | = |z-w|/|\Im(z)| \).
Indeed, this holds since
\begin{align*}
    g(x^*)
    &= \frac{\Re(z) + \frac{\Im(z)^2}{\Re(z)-w} - w}{\sqrt{ \left( \Re(z) - \left( \Re(z) + \frac{\Im(z)^2}{\Re(z)-w} \right)\right)^2 + \Im(z)^2 }}
    \\&= \operatorname{sign}(\Re(z)-w)\frac{(\Re(z)-w)^2 + \Im(z)^2}{|\Im(z)|\sqrt{(\Re(z)-w)^2 + \Im(z)^2}}
    \\&= \operatorname{sign}(\Re(z)-w) \left| \frac{z-w}{\Im(z)} \right|.
\end{align*}
}
\fi
\end{proof}

\begin{lemma}
\label{thm:level_sets}
Fix \( r >0 \), let  \( \mathcal{D}(c,t) \) be the disc in the complex plane centered at \( c \) with radius \( t \geq 0\), and define
\begin{align*}
    X_r = \bigcup_{ x \in [a,b]  } \mathcal{D}\left( x, \frac{|x-w|}{r} \right).
\end{align*}
%where \( \mathcal{D}(c,t) \) is the disc centered at \( c \) with radius \( t \geq 0\).
Then for \( z \in \mathbb{C} \setminus X_r \), we have
\begin{align*}
    \|h_{w,z}\|_{[a,b]} 
    %= \max_{x \in \mathcal{I}(\vec{A})} \left| \frac{x-w}{x-z} \right| 
    \leq r.
\end{align*}
In particular, if \( z \) is on the boundary of \( X_r \), then \( \|h_{w,z}\|_{[a,b]}  = r \).
\end{lemma}

\begin{proof}
Let \( z \in \mathbb{C}\setminus X_r \) and pick any \( x \in [a,b] \).
 Since \( z \not\in \mathcal{D}(x,|x-w|/r) \) it follows that \( |z-x| > |x-w|/r \) and therefore \( |h_{w,z}(x)| = |x-w|/|x-z| < r \).
Maximizing over \( x \) yields the result.

If \( z\) is on the boundary of \( X_r \),  then for some \( x \in [a,b] \), \( |z-x| = |x-w|/r \), which means that for this \( x \), \(|h_{w,z}(x)| = r \).
\end{proof}
Note that if \( r \leq 1 \) and \( w \in \R\setminus[a,b] \), then the region described in \cref{thm:level_sets} is simply a disc about \( b \) if \( w < a\) or a disc about \( a \) if \( w > b \).  If \( r > 1 \) and \( w \) is real, then the region described is that in the discs about \( a \) and \( b \) and between the two external tangents to these two discs.

\subsection{A priori bounds}

We can use \cref{thm:err_int} to give a priori bounds, as long as we choose \( S_0 \) and \( S_i \), \( i =1, \ldots, k \) independently of $\vec b$ (and in turn $\vec T_k$). %\Cam{In intro we complain that certain bounds depend on the spectrum making them bad a priori bounds. Should we reword this a bit?}

The simplest possibility is to take \( S_0 = S_i = \mathcal{I}(\vec{A}) \).
In this case, as an immediate consequence of \cref{thm:err_int,thm:level_sets} we have the following a priori bound,
\begin{corollary}
\label{thm:disk}
    Suppose that for some \( w < \lambda_{\textup{min}}(\vec{A}) \), \( f \) is analytic in a neighborhood of \( \mathcal{D}(\lambda_{\textup{max}}(\vec{A}),\lambda_{\textup{max}}(\vec{A})-w) \).
%on \( X_1 = \mathcal{D}(\lambda_{\text{max}}(\vec{A}), \lambda_{\text{max}}(\vec{A})) \).
    Then, taking \( \Gamma \) to be the boundary of this disk, 
%\( \Gamma = \{ r \left( \exp(i t) + r \right): r = \lambda_{\text{max}}, t \in [0,\pi] \} \),
   \begin{align*}
        \| f(\vec{A})\vec{b} - \lan_k(f) \|
        &\leq \left( \frac{1}{2\pi} \oint_{\Gamma} |f(z)| \cdot |\d{z}| \right) \| \err_k(w) \|
        \\&\leq \left( (\lambda_{\textup{max}}(\vec{A})-w) \: \max_{z\in\Gamma}|f(z)|\right) \| \err_k(w) \|.
    \end{align*}
\end{corollary}
\begin{proof}
To obtain the first inequality  observe that \cref{thm:level_sets} with \( [a,b] = \mathcal{I}(\vec A) \) implies \( \|h_{w,z}\|_{\mathcal{I}(\vec{A})} = 1 \) on this contour.
The second inequality follows since the length of \( \Gamma \) is \( 2 \pi (\lambda_{\textup{max}}(\vec{A})-w) \).
\end{proof}
%We remark that taking \( \Gamma \) as the \( \|h_{w,z}\|_{\mathcal{I}(\vec{A})} = 1 \) level curve gives a somewhat shorter contour which can be used in place of the disk used in \cref{thm:disk}.
%This will often yield a slightly better bound, albeit at the cost of a more complicated parameteriztion of the contour.
This bound is closely related to \cite[Theorem 6.6]{frommer_guttel_schweitzer_14a} which bounds the error in Lanczos-FA for Stieltjes
functions in terms of the error in the Lanczos approximation for a certain linear system.

    Using that \( \err_0(w) = (\vec{A}-w\vec{I})^{-1} \vec{b} \), we can rewrite \cref{thm:disk} as
\begin{align*}
    \frac{ \| f(\vec A) \vec b - \lan_k(f) \|_2}{\| f(\vec A)\vec b \|_2} \le \max_{z \in \Gamma} |f(z)| \cdot  \frac{(\lambda_{\textup{max}}(\vec{A})- w) \|(\vec{A} - w\vec{I})^{-1} \vec b \|_2}{\| f(\vec A)\vec b \|_2} \cdot \frac{\|\err_k(w)\|_2}{\| \err_0(w) \|_2}. 
\end{align*}
This can be used to obtain simple relative error bounds for many functions. 
For instance, suppose $\vec{A}$ is positive definite,  $f(x) = x^{-q}$ for \( q > 1 \), and $w = c \lambda_{\text{min}}$ for $c\in (0,1)$.
Then $\max_{z\in\Gamma} |z^{-q}| = w^{-q} = c^{-q} \lambda_{\textup{min}}(\vec{A})^{-q}$, \( \| (\vec{A} - w\vec{I})^{-1} \vec{b} \|_2 \leq (\lambda_{\textup{min}}(\vec{A})-w)^{-1} \| \vec{b} \| \) and \( \| \vec{A}^{-q} \vec{b} \|_2 \geq \lambda_{\textup{max}}(\vec{A})^{-q} \| \vec{b} \| \).
We then have the bound\footnote{Slightly stronger bounds can be obtained by bounding \( \| (\vec{A}-w\vec{I})^{-1} \vec{b} \|_2 / \| \vec{A}^{-q} \vec{b} \|_2 \) directly, rather than bounding the numerator and denominator separately.} 
\begin{align*}
    \frac{ \| \vec{A}^{-q} \vec b - \lan_k(f) \|_2}{\| \vec{A}^{-q}\vec b \|_2} 
    %&\le \max_{z \in \Gamma} |z^{-q}| \cdot  \frac{(\lambda_{\textup{max}}(\vec{A})- w) \cdot \|(\vec{A} - w\vec{I})^{-1} \vec b \|_2}{\| \vec{A}^{-q} \vec b \|_2} \cdot \frac{\|\err_k(w)\|_2}{\| \err_0(w) \|_2} 
    & \leq 
    c^{-q} \kappa(\vec{A})^{q} \kappa( \vec{A} - w\vec{I})  \frac{\|\err_k(w)\|_2}{\| \err_0(w) \|_2} .
\end{align*}

\Cref{thm:disk} and the above bound provide simple reductions to the error of solving a positive definite linear system involving \( \vec{A} - w\vec{I} \) using Lanczos. 
However, these bounds may be a significant overestimate in practice.
In particular, for any \( k>1 \), \cref{eqn:dkwz_qwz} cannot be sharp due to the fact that \( \|h_{w,z}\|_{\mathcal{I}(\vec{A})} = \sup_{x \in \mathcal{I}(\vec{A})} |h_{w,z}(x)| \) cannot be attained at every eigenvalue of \( \vec{T}_k \).
In fact, for most values \( \lambda_i ( \vec{T}_k ) \) and most points \( z \in \Gamma \), we expect \( | h_{w,z}( \lambda_i ( \vec{T}_k ) ) | \ll \|h_{w,z}\|_{\mathcal{I}(\vec{A})} \).
\Cref{fig:Q_wz_D_kwz_levels} shows sample level curves for \( \|h_{w,z}\|_{\mathcal{I}(\vec{A})} / |\det(h_{w,z}(\vec{T}_k))|^{1/k} \) which illustrate  the slackness in the bound.

To derive sharper a priori bounds, there are several approaches. 
If more information is known about the eigenvalue distribution of \( \vec{A} \), then the \( S_i \) can be chosen based on this information.
For example, similarly to \cref{eqn:triangle_ineq2}, it is possible to exploit the interlacing property of the eigenvalues of \( \vec{T}_k \).
\begin{example}
Suppose \( \vec{A} \) has eigenvalues in \( [0,1] \) with a single eigenvalue at \( \kappa > 1 \).  Assume \( w \leq 0 \).
Then there is at most one eigenvalue of \( \vec{T}_k \) in \( [1,\kappa] \) so in \cref{thm:err_int} we can pick $S_i = [0,1]$ for $i = 1,\ldots, k-1$ and $S_k = [0,\kappa]$. We have
\begin{align*}
    |\!\det(h_{w,z}(\vec{T}_k))| = \left| \prod_{i=1}^{k} h_{w,z}(\lambda_i(\vec{T}_k)) \right| \leq
      \left(\| h_{w,z} \|_{[0,1]}\right)^{k-1} \|h_{w,z}\|_{[0,\kappa]}.
\end{align*}
If \( z \) is near to \( \kappa \) then \( \| h_{w,z} \|_{[0,1]} \) may be much smaller than \( \| h_{w,z} \|_{[0,\kappa]} \).
\end{example}

Second, the contour \( \Gamma \) can be chosen to try to reduce the slackness in \cref{eqn:dkwz_qwz}.
Intuitively, the slackness is exacerbated when \( z \in \Gamma \) is close to \( S_i \) but far from \( \lambda_i(\vec{T}_k) \).
For instance, for any \( k>1 \),
\begin{align*}
\lim_{|z|\to\infty} \frac{\|h_{w,z}\|_{\mathcal{I}(\vec{A})}^k}{|\det(h_{w,z}(\vec{T}_k))|} \to 1 
,&&\text{and }&&
\forall \lambda \in \mathcal{I}(\vec{A}),~ \lim_{z\to\lambda} \frac{\|h_{w,z}\|_{\mathcal{I}(\vec{A})}^k}{|\det(h_{w,z}(\vec{T}_k))|} \to \infty.
\end{align*}
This behavior is also observed in \cref{fig:Q_wz_D_kwz_levels}.
\begin{figure}[ht]
    \begin{subfigure}{.48\textwidth}\centering
        \includegraphics[width=\textwidth]{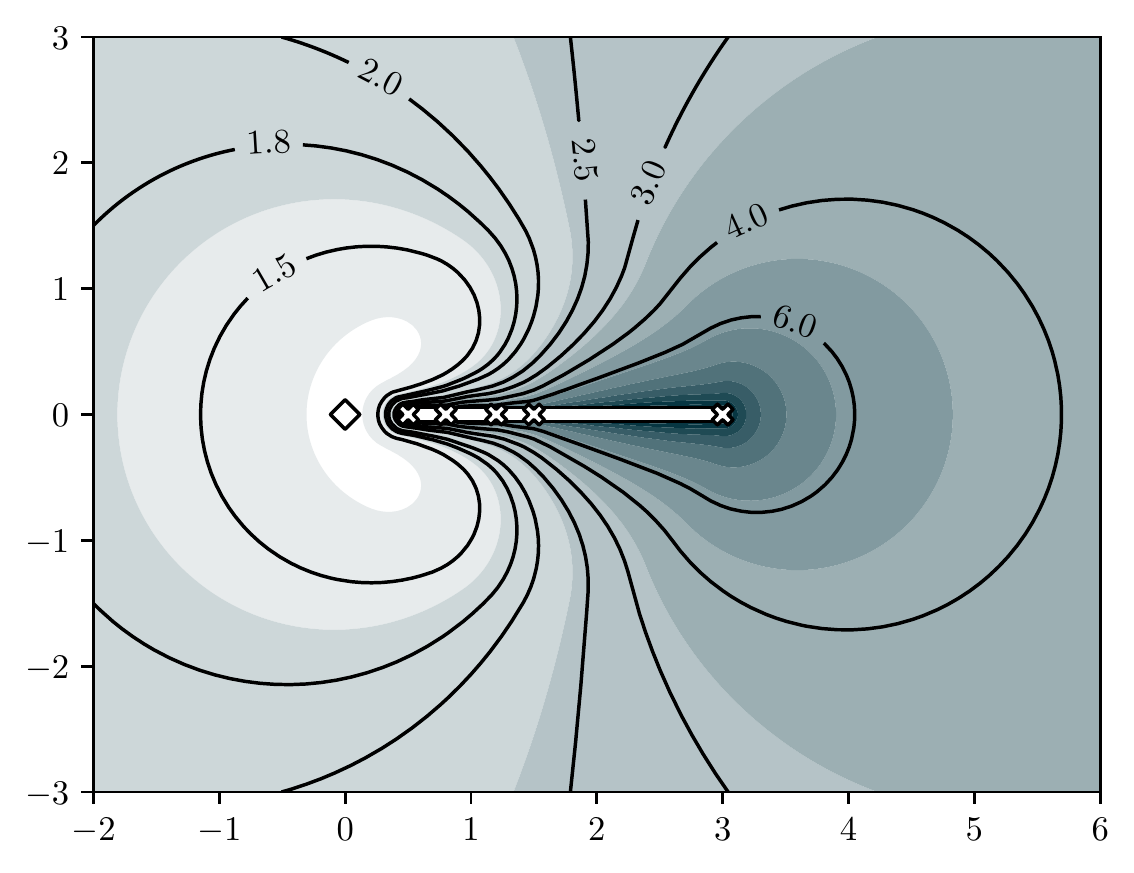}
    \end{subfigure}\hfill 
    \begin{subfigure}{.48\textwidth}\centering
        \includegraphics[width=\textwidth]{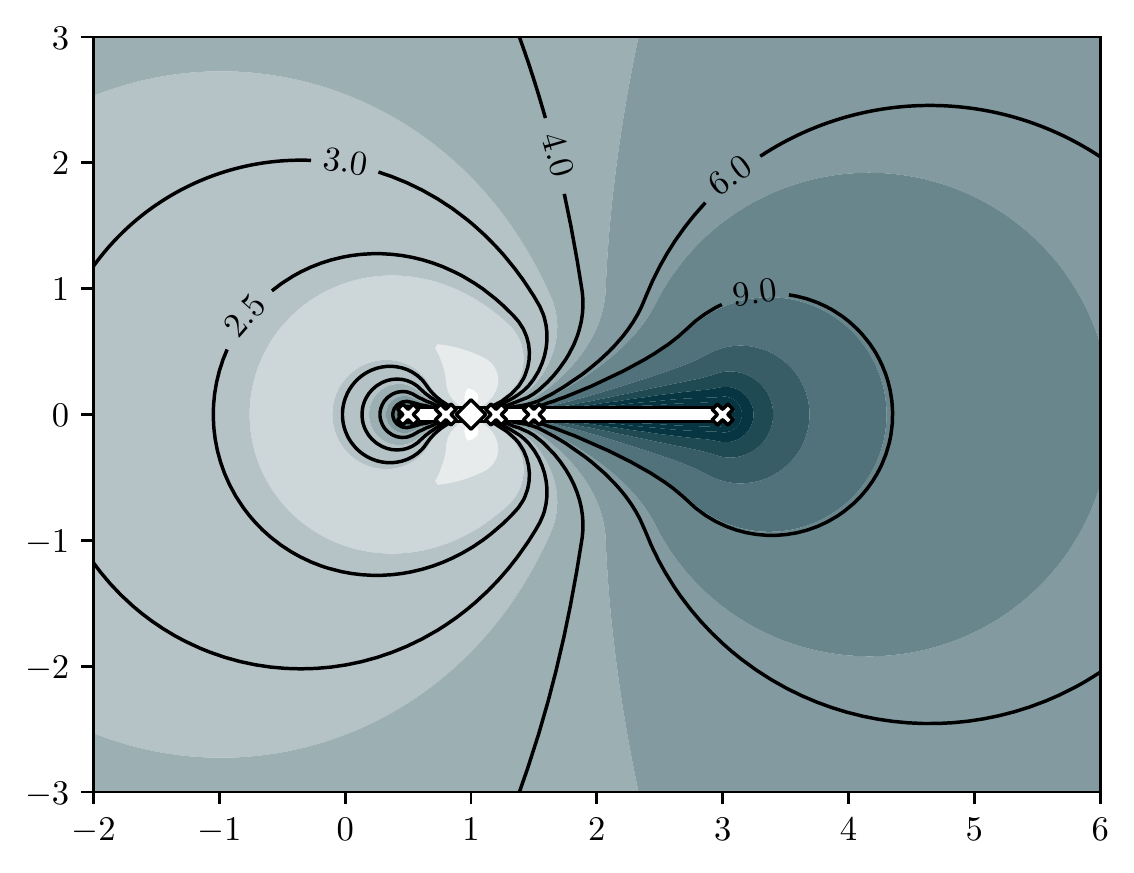}
    \end{subfigure}
    \caption{
        Contour plot of \( \|h_{w,z}\|_{\mathcal{I}(\vec{A})} / |\!\det(h_{w,z}(\vec{T}_k))|^{1/k} \) as a function of $z \in \mathbb{C}$ for a synthetic example with \( \mathcal{I}(\vec{A}) = [0.5,3] \) and \( \Lambda(\vec{T}_k) = \{ 0.5,0.8,1.2,1.5,3 \} \) ($k=5$).
        Here \( w \) is indicated by the white diamond ({\protect\includegraphics[scale=.7]{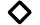}}) and the eigenvalues of \( \vec{T}_k \) are indicated by white x's ({\protect\includegraphics[scale=.7]{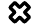}}).
        Larger slackness in \cref{eqn:dkwz_qwz} corresponds to darker regions. 
    }
    \label{fig:Q_wz_D_kwz_levels}
\end{figure}

These observations suggest that we should pick \( \Gamma \) to be far from the spectrum of \( \vec{A} \).
Of course, we are constrained by properties of \( f \) such as branch cuts and singularities.
Moreover, certain contours may increase the slackness in \cref{thm:err_int} itself. 
These considerations are discussed further in \cref{ex:sqrt_contours}.

\subsection{A posteriori error bounds}

After the Lanczos factorization \cref{eqn:lanczos_factorization} has been computed, \( \vec{T}_k \) is known and \( \Lambda(\vec{T}_k) \) can be cheaply computed. 
Thus, in \cref{thm:err_int} we can take \( S_i = \{ \lambda_i(\vec{T}_k) \} \) for \( i =1, \ldots, k \), which is the best possible choice. 
In this case \cref{eqn:dkwz_qwz} is an equality and \( \det(h_{w,z}(\vec{T}_k)) = \det(\vec{T}_k-w)/\det(\vec{T}_k-z) \) can be computed via tridiagonal determinant formulas rather than using the eigenvalues of \( \vec{T}_k \).

If \( \mathcal{I}(\vec{A}) \) is not known, the extreme Ritz values \( \lambda_{\text{min}}(\vec{T}_k) \) and \( \lambda_{\text{max}}(\vec{T}_k) \) can be used to estimate the extreme eigenvalues of \( \vec{A} \) \cite{kuczyski_wozniakowski_92,parlett_simon_stringer_82}.
All together, this means that it is not difficult to efficiently obtain accurate estimates of the bound from \cref{thm:err_int}.

\subsection{Numerical computation of integrals}
\label{sec:numerical_integral}
Typically, to produce an a priori or a posteriori error bound, the integral term in \cref{thm:err_int} must be computed numerically.
Consider a discretization of the integral
\begin{align*}
    f(\vec{A}) = -\frac{1}{2\pi i} \oint_{\Gamma} f(z) (\vec{A}-z\vec{I})^{-1} \d{z}
\end{align*}
using nodes \( z_i \) and weights \( w_i \), \( i=1,2,...,q \). 
This yields a rational matrix function
\begin{align*}
    \label{eqn:rational_func}
    r_q(\vec{A}) := -\frac{1}{2\pi i} \sum_{i=1}^{q}  w_i f(z_i) (\vec{A}-z_i\vec{I})^{-1}. 
\end{align*}
Using the triangle inequality, we can write
\begin{align}
\hspace{5em}&\hspace{-5em}\nonumber
\| f(\vec{A}) \vec{b} - \lan_k(f) \|
\\&\leq 
\nonumber
\| f(\vec{A}) \vec{b} - r_q(\vec{A}) \vec{b} \| + \| r_q(\vec{A}) \vec{b} - \lan_k(r_q)) \| + \| \lan_k(r_q) - \lan_k(f) \|
\\&\leq 2  \left( \max_{x \in \Lambda(\vec{A}) \cup \Lambda(\vec{T}_k)} | f(x) - r_q(x) | \right) \|\vec{b}\|  + \| r_q(\vec{A}) \vec{b} - \lan_k(r_q) \|. \label{eqn:triangle_ineq_rational}
%\nonumber
%\\& \leq \underbrace{2 \|\vec{b}\| \left( \max_{x \in \mathcal{I}(\vec{A})} | f(x) - r(x) | \right)}_{\text{approximation error}} 
%    + \underbrace{\vphantom{ \bigg| } \| r(\vec{A}) \vec{b} - \lan_k(r) \|}_{\text{application error}}.
%    \label{eqn:triangle_ineq}
\end{align}
Now, observe that analogous to \cref{thm:err_int},
\begin{align}
    \label[ineq]{eqn:err_int_rational}
    \| r_q(\vec{A})\vec{b} - \lan_k(r_q) \| \leq \left( \frac{1}{2\pi} \sum_{i=1}^{q} w_i \cdot |f(z_i)| \cdot \left( \prod_{i=1}^{k} \|h_{w,z}\|_{S_i} \right) \cdot \|h_{w,z}\|_{S_0} \right) \| \err_k(w) \|.
\end{align}
If we use the same nodes and weights to evaluate the integral term in \cref{thm:err_int}, we obtain exactly the expression on the right hand side of \cref{eqn:err_int_rational}.
Thus, this discretization of \cref{thm:err_int} is a true upper bound for the Lanczos-FA error to within an additive error of size equal to twice the approximation error of \( r(x) \) to \( f(x) \) on \( \Lambda(\vec{A}) \cup \Lambda(\vec{T}_k) \) times \( \| \vec{b} \| \).
In many cases, we expect exponential convergence of \( r_q \) to \( f \), which implies that this term can be made less than any desired value $\epsilon > 0$ using a number of quadrature nodes that grows only as the logarithm of \( \epsilon^{-1} \) \cite{hale_higham_trefethen_08,trefethen_weideman_14}.

We note that fast convergence of \( r_q \) to \( f \) suggests that, instead of applying Lanczos-FA, we can approximate \( f(\vec{A}) \vec{b} \) by first finding \( r_q \) and then solving a small number of linear systems \( (\vec{A} - z_i \vec{I}) \vec{x}_i = \vec{b} \) to compute \( r_q(\vec{A})\vec{b} \).
Solving these systems with any fast linear system solver yields an algorithm for approximating \( f(\vec{A})\vec{b} \) inheriting, up to logarithmic factors in the error tolerance, the same convergence guarantees as the linear system solvers used.
A recent example of this approach is found in \cite{jin_sidford_19} which uses a modified version of stochastic variance reduced gradient (SVRG) to obtain a nearly input sparsity time algorithm for \( f(\vec{A})\vec{b} \) when \( f \) corresponds to principal component projection or regression.

A range of work suggests using a Krylov subspace method and the shift invariance of the Krylov subspace to solve these systems and compute $r_q(\vec A)\vec b$ explicitly.
This was studied in \cite{frommer_kahl_lippert_rittich_13,frommer_simoncini_09} for the Lanczos method, and in \cite{pleiss_jankowiak_eriksson_damle_gardner_20} for MINRES, the latter of which uses the results of \cite{hale_higham_trefethen_08} to determine the quadrature nodes and weights. 
However, as the above argument demonstrates, the limit of the Lanczos-based approximation as the discretization becomes finer is simply the Lanczos-FA approximation to \( f(\vec{A}) \vec{b} \).
Therefore, there is no clear advantage to such an approach over Lanczos-FA in terms of the convergence properties, unless preconditioning is used.
.

On the other hand, there are some advantages to these approaches in terms of computation.
Indeed, Krylov solvers for symmetric/Hermitian linear systems require just $O(n)$ storage; i.e. they do not require more storage as more iterations are taken. 
A naive implementation of Lanczos-FA requires $O(kn)$ storage, and while Lanczos-FA can be implemented to use \( O(n) \) storage by taking two passes, this has the effect of doubling the number of matrix-vector products required.
See \cite{guttel_schweitzer_21} for a recent overview of limited-memory Krylov subspace methods.

\section{Examples and numerical verification}
\label{sec:examples}

We next present examples in which we apply  \cref{thm:err_int} to give a posteriori and a priori error bounds for approximating common matrix functions with Lanczos-FA. These examples illustrate the general approaches to applying \cref{thm:err_int} described in \cref{sec:results}. 
All integrals are computed either analytically or using SciPy's \texttt{integrate.quad} which is a wrapper for QUADPACK routines. 

In all cases, we exactly compute the $\| \err_k(w) \|$ term in the bounds. In practice, one would bound  this quantity a priori or a posteriori using existing results on bounding the Lanczos error for linear system solves. By computing the error exactly, we separate any looseness due to our bounds from any looseness due to an applied bound on  $\| \err_k(w) \|$.

\begin{example}[Matrix square root]
\label{ex:sqrt_contours}
Let \( \vec{A} \) be positive definite and  \( f(x) = \sqrt{x} \).
Perhaps the simplest bound is obtained by using \cref{thm:err_int} with $w = 0$, \( S_i = \mathcal{I}(\vec{A}) \) and \( \Gamma \) chosen as the boundary of the disk \( \mathcal{D}(\lambda_{\text{max}}(\vec{A}), \lambda_{\text{max}}(\vec{A}) ) \).%or as the \( \|h_{w,z}\|_{\mathcal{I}(\vec{A})} = 1 \) level curve of \( \|h_{w,z}\|_{\mathcal{I}(\vec{A})} \).
We then obtain a bound via \cref{thm:disk}. 
However, this bound may be loose -- note that except through $\|\err_k(w)\|$, it does not depend on the number of iterations $k$.  Thus it cannot establish convergence at a rate faster than that of solving a linear system with coefficient matrix \( \vec{A} \).

\begin{figure}[h]
\centering
\begin{subfigure}{.25\textwidth}\centering
\includegraphics[scale=.7]{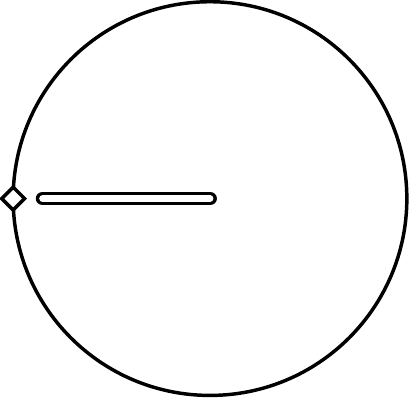}
\caption{circle contour}
\label{fig:circle}
\end{subfigure}
\hfill
\begin{subfigure}{.25\textwidth}\centering
\includegraphics[scale=.7]{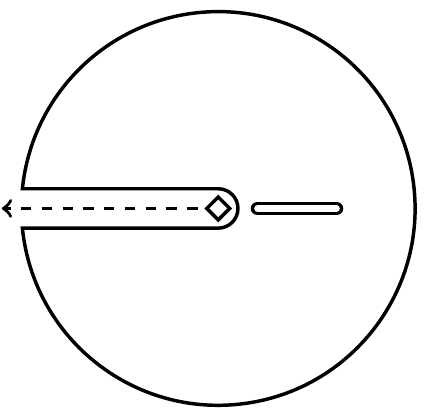}
\caption{Pac-Man contour}
\label{fig:pacman}
\end{subfigure}
\hfill
\begin{subfigure}{.35\textwidth}\centering
\includegraphics[scale=.7]{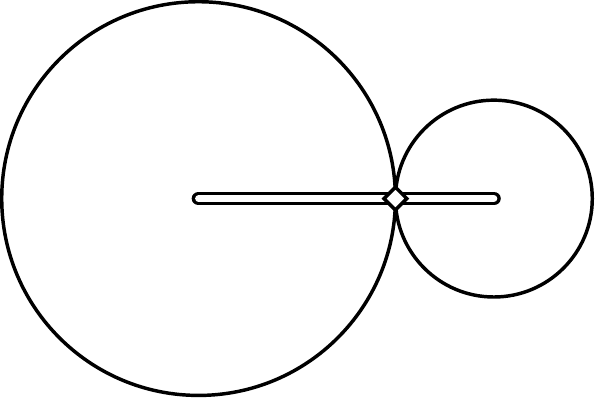}
\caption{double circle contour}
\label{fig:double_circle}
\end{subfigure}
\caption{
Circle, Pac-Man and double circle contours described in \cref{ex:sqrt_contours,ex:cif_pw} respectively.
All three figures show \( \mathcal{I}(\vec{A}) \) ({\protect \includegraphics[scale=.7]{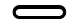}}) and \( w \) ({\protect\includegraphics[scale=.7]{imgs/legend/diamond.pdf}}).
}
\label{fig:contours}
\end{figure}

Keeping $w = 0$, we can obtain tighter bounds by letting $\Gamma$ be a ``Pac-Man'' like contour that consists of a large circle about the origin of radius \( R \) with a small circular cutout of radius \( r \) that excludes the origin and a small strip cutout to exclude the negative real axis. %, as pictured in \cref{fig:Pac-man}.
    That is, as shown in \cref{fig:pacman}, the boundary of the set,
\begin{align*}
    \mathcal{D}(0,R) \setminus ( \{ z : \Re(z) \leq 0, |\Im(z)| < r \} \cup \mathcal{D}(0,r) ).
\end{align*}

\begin{figure}[ht]
    \begin{subfigure}{.48\textwidth}\centering
        \includegraphics[width=\textwidth]{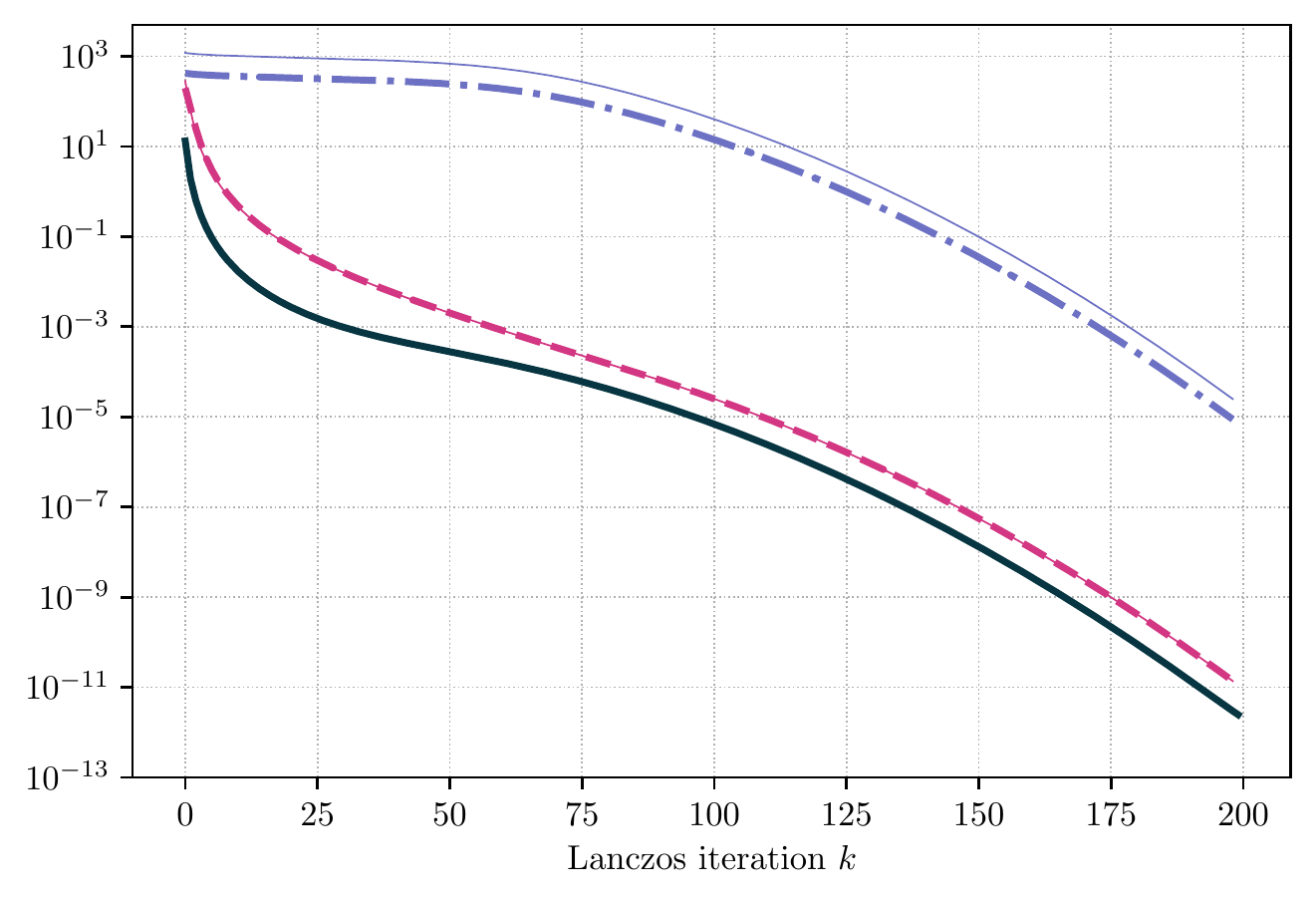}
        \caption{circular contour (\( c = r = \lambda_{\text{max}}(\vec{A}) \))}
    \end{subfigure}\hfill 
    \begin{subfigure}{.48\textwidth}\centering
        \includegraphics[width=\textwidth]{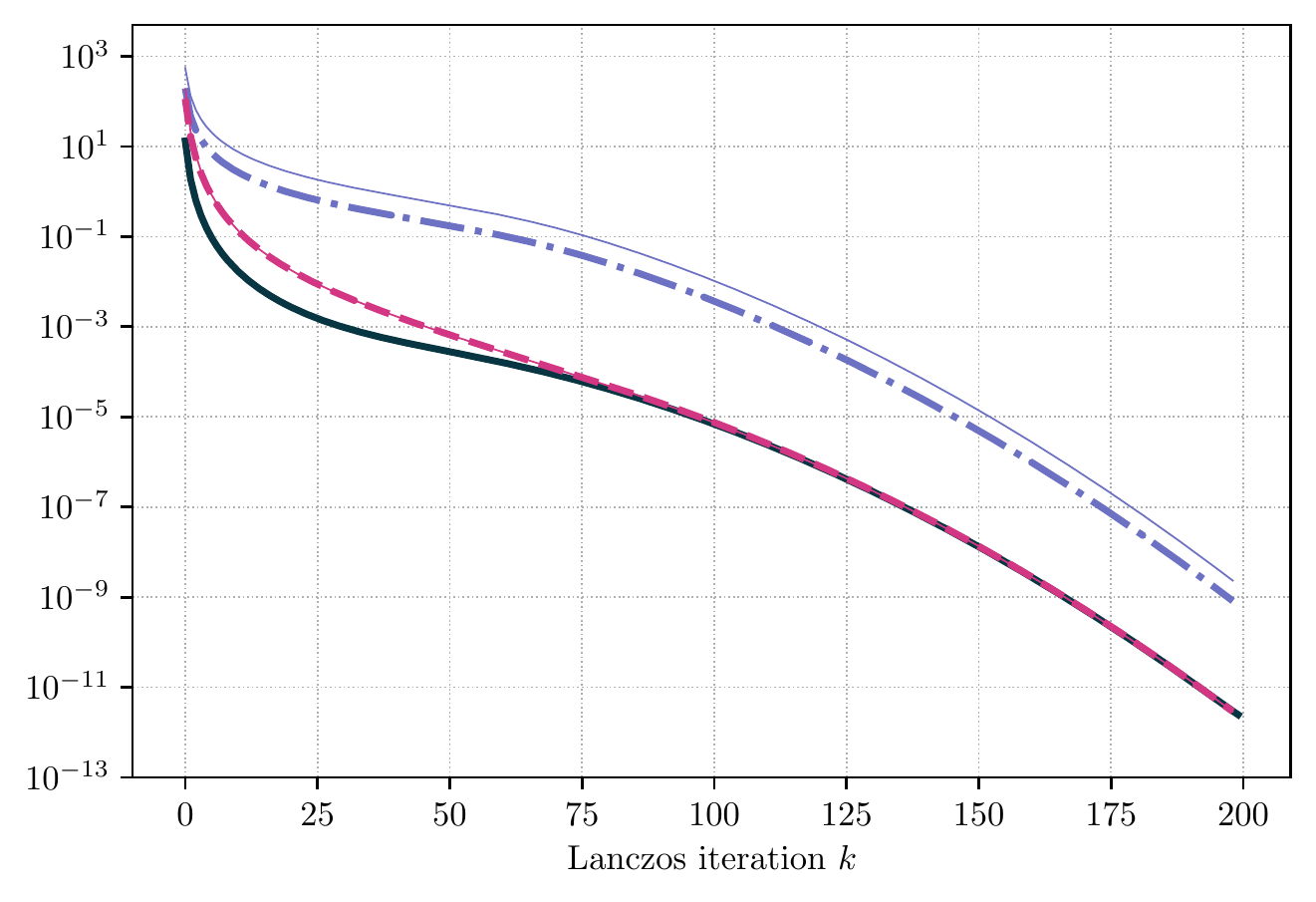}
        \caption{Pac-Man contour ($r \rightarrow 0$, $R \rightarrow \infty$)}
    \end{subfigure}
    \caption{
        \( \vec{A} \)-norm error bounds for \( f(x) = \sqrt{x} \) where \( \vec{A} \) has \( n=1000 \) eigenvalues spaced uniformly in \( [10^{-2},10^2] \).
        % as in \cite[Figure 1]{frommer_guttel_schweitzer_14}.
        \emph{Legend}: 
        True Lanczos-FA error \( \| f(\vec{A})\vec{b} - \lan_k(f) \| \) ({\protect\raisebox{0mm}{\protect\includegraphics[scale=.7]{imgs/legend/solid_blue.pdf}}}).
        A priori bounds obtained by using \cref{thm:err_int} with \( S_0 = S_i = \mathcal{I}(\vec{A})  \) ({\protect\raisebox{0mm}{\protect\includegraphics[scale=.7]{imgs/legend/dashdot_purple.pdf}}}) and 
        \( S_0 = S_i = \tilde{\mathcal{I}}(\vec{A}) = [ \lambda_{\text{min}}(\vec{A})/2, 2 \lambda_{\text{max}}(\vec{A})] \)
        ({\protect\raisebox{0mm}{\protect\includegraphics[scale=.7]{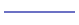}}}).
        A posteriori bounds obtained by using \cref{thm:err_int} with \( S_0 = \mathcal{I}(\vec{A}) \), \( S_i = \{ \lambda_i(\vec{T}_k) \} \) ({\protect\raisebox{0mm}{\protect\includegraphics[scale=.7]{imgs/legend/dash_pink.pdf}}}), and
        \( S_0 = \tilde{\mathcal{I}}(\vec{A}) \), \( S_i = \{ \lambda_i(\vec{T}_k) \} \)
        ({\protect\raisebox{0mm}{\protect\includegraphics[scale=.7]{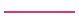}}}).
        Observe that using the wider interval \( \tilde{\mathcal{I}}(\vec{A}) \) has very little effect on both the a priori and a posteriori bounds. Also observe that the a posteriori bounds closely match the actual convergence of Lanczos-FA. 
    }
    \label{fig:sqrt_error_bounds}
\end{figure}
As the outer radius \( R \rightarrow \infty \), the integral over the large circular arc goes to \( 0 \) since \( \|h_{w,z}\|_{\mathcal{I}(\vec{A})} = O(R^{-1}) \), \( | f(z) | = O(R^{1/2}) \), and the length of the circular arc is on the order of \( R \). 
Thus, the product \( f(z) (\| h_{w,z} \|_{\mathcal{I}(\vec{A})})^{k+1} \) goes to \( 0 \) as \( R \rightarrow \infty \), for all \( k \geq 1 \).  
Similarly, as \( r\to 0 \), the length of the small arc goes to zero.
Therefore, we need only consider the contributions to the integral on \( [ -R \pm ir, \pm ir] \) in the limit \( R\to \infty, r\to 0 \).

In this case, when $S_i = \mathcal{I}(\vec A)$ for all $i$, we can compute the value of the integral term in \cref{thm:err_int} analytically. We have
\begin{align*}
\| f( \vec{A} ) \vec{b} - \lan_k (f) \| 
&\leq \left(\frac{1}{2 \pi} \int_{-\infty}^0 | (x \pm 0 i)^{1/2} | \cdot  \| h_{w,x \pm 0 i} \|_{\mathcal{I}(\vec{A})}^{k+1} \,\d{x}  \right) \| \err_k \| 
\\&= \left(\frac{1}{2 \pi} \int_{-\infty}^0 | x \pm 0 i|^{1/2}\frac{ \lambda_{\text{max}} ( \vec{A})^{k+1}}{( \lambda_{\text{max}} ( \vec{A} ) - x )^{k+1}}\,\d{x}   \right) \| \err_k \| 
\\&= \left(\frac{1}{\pi} \lambda_{\text{max}} ( \vec{A} )^{k+1}  \int_0^{\infty} \frac{y^{1/2}}{( \lambda_{\text{max}} ( \vec{A} ) + y )^{k+1}}\d{y} \right) \| \err_k \| 
\\&= \left( \frac{\lambda_{\text{max}}(\vec{A})^{3/2}}{2\sqrt{\pi}} \frac{\Gamma(k-1/2)}{\Gamma(k+1)}  \right) \| \err_k \|,
%\frac{1}{2 \pi} \int_{-R}^0 | (x \pm ir )^{1/2} |\cdot| \det ( h_{w,x \pm ir} (\vec{T}_k)) |\cdot\| h_{w,x \pm ir} ( \vec{A} ) \|\,\d{x} 
%\frac{1}{2 \pi} \int_{-R}^0 | (x \pm ir )^{1/2} |\cdot| \det ( h_{w,x \pm ir} (\vec{T}_k)) |\cdot\| h_{w,x \pm ir} ( \vec{A} ) \|\,\d{x} 
\end{align*}
where we have made the change of variable \( y = -x \). 
Note that
\begin{align*}
    \lim_{k \to \infty }k^{3/2}  \frac{\Gamma(k-1/2)}{\Gamma(k+1)} = 1.
\end{align*}
This proves that \( \lan_k ( \sqrt{\cdot}) \) converges somewhat faster than the Lanczos algorithm applied to the corresponding linear system \( \vec{A} \vec{x} = \vec{b} \).

In \cref{fig:sqrt_error_bounds}, we plot the bounds from \cref{thm:err_int} for the circular and Pac-Man contours described above.
For both contours we consider \( S_i = \mathcal{I}(\vec{A}) \) for all \( i \), as well as bounds based on an overestimate of this interval,  \( S_i = \tilde{\mathcal{I}}(\vec{A}) \) where 
\(
    \tilde{\mathcal{I}}(\vec{A}) = [ \lambda_{\text{min}}(\vec{A})/2, 2 \lambda_{\text{max}}(\vec{A})].
\)
This provides some sense of how sensitive the bounds are to the choice of \( S_i \) when $S_i$ is a single interval. For a posteriori bounds, we set \( S_i \) to $\{\lambda_i(\vec T_k)\}$ for \( i > 0 \).

We remark that the bounds from \cref{thm:err_int} are upper bounds for \cref{eqn:integral_error} which implies that the slackness of \cref{eqn:integral_error} is relatively small.
This suggests that the roughly 6 orders of magnitude improvement in \cref{thm:err_int} when moving from the circular contour to the Pac-Man contour is primarily due to reducing the slackness in \cref{eqn:dkwz_qwz}, aligning with our intuition.
\end{example}

\begin{figure}[htb]
    \begin{subfigure}[t]{.48\textwidth}\centering
        \includegraphics[width=\textwidth]{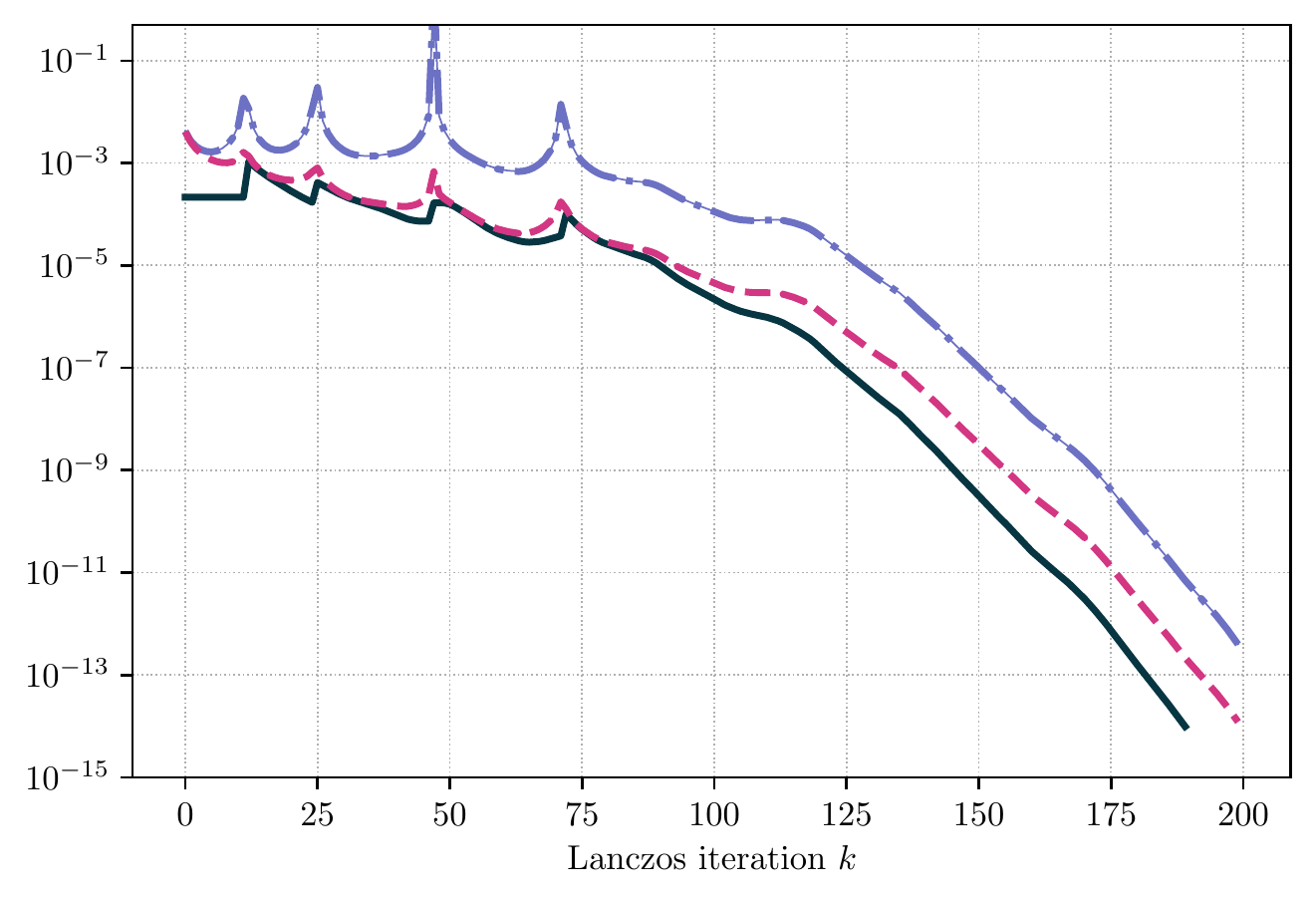}
        \caption{\((\vec{A}-w\vec{I})^2 \)-norm for \( f(x) = \operatorname{step}(x-a)/x \) where \( \vec{A} = \vec{X}\vec{X}^\cT \) and the entries of \( \vec{X} \in\mathbb{R}^{n,2n} \) are independent Gaussians with mean zero and variance \( 1/2n \) with \( n = 3000 \). 
        We set \( a=0.99 \lambda_{\text{max}}(\vec{A}) \) so that there are roughly $5$ eigenvalues above \( a \).}
    \end{subfigure}\hfill 
    \begin{subfigure}[t]{.48\textwidth}\centering
        \includegraphics[width=\textwidth]{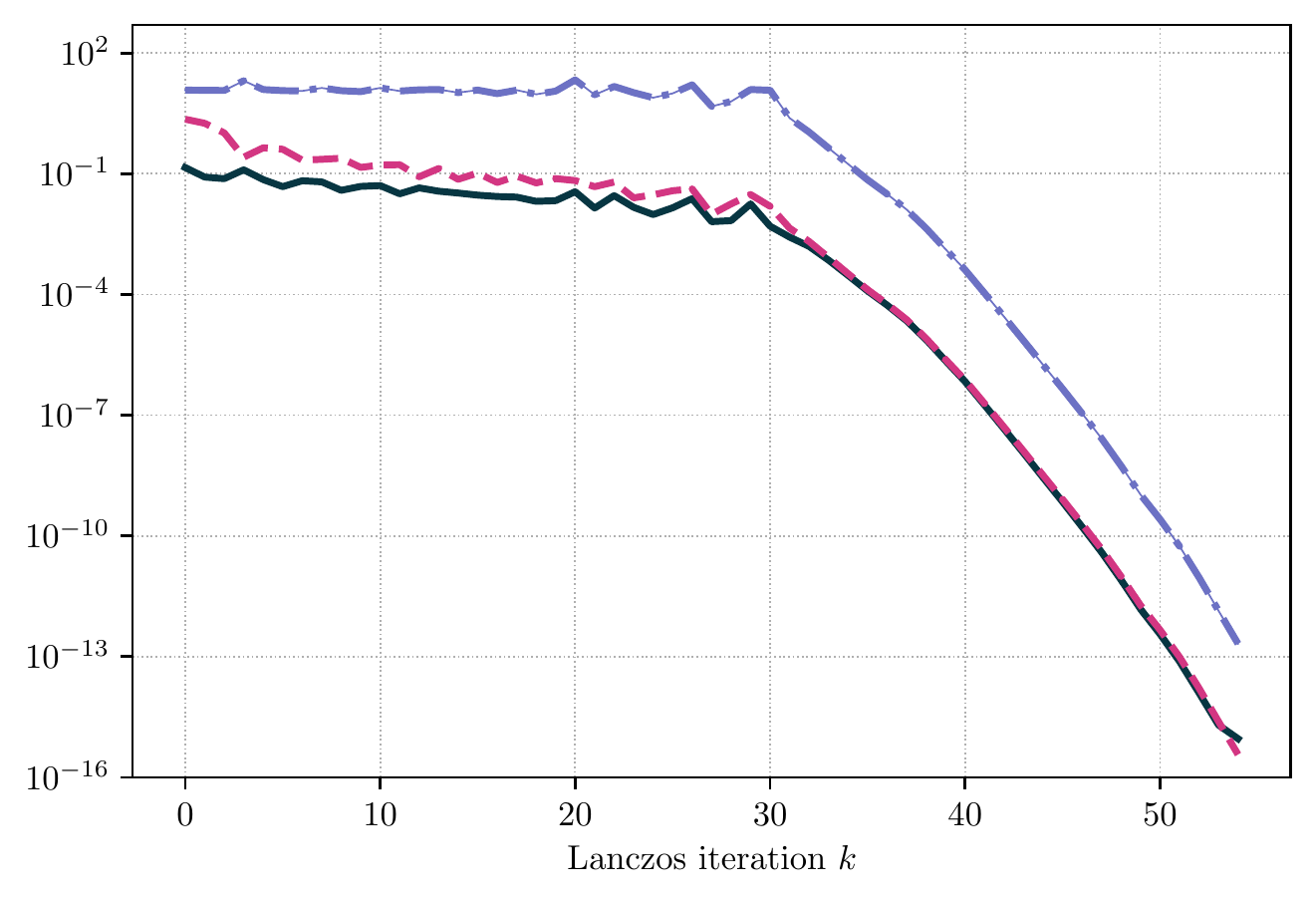}
        \caption{\( 2 \)-norm for \( f(x) = \operatorname{step}(x-a) \) where \( \vec{A} \) is the MNIST training data \cite{lecun_cortes_burges_10} covariance matrix and \( a=0.15 \lambda_{\text{max}}(\vec A) \) so that there are 16 eigenvalues above \( a \). %\Cam{Should cite this data source}
        }
    \end{subfigure}
    \caption{
        Bounds for piece-wise analytic functions using the double-circle contour described in \cref{ex:cif_pw}.
        \emph{Legend}: 
        True Lanczos-FA error ({\protect\raisebox{0mm}{\protect\includegraphics[scale=.7]{imgs/legend/solid_blue.pdf}}}). 
        A priori bounds obtained by using \cref{thm:err_int} with \( S_0 = S_i = \mathcal{I}(\vec{A}) \) ({\protect\raisebox{0mm}{\protect\includegraphics[scale=.7]{imgs/legend/dashdot_purple.pdf}}}) or \cref{eqn:ex_cif_pw_a_priori2}  
        ({\protect\raisebox{0mm}{\protect\includegraphics[scale=.7]{imgs/legend/thin_purple.pdf}}}) with \( S_0 = S_i = \mathcal{I}(\vec{A}) \).
        Note that these curves are on top of one another suggesting there is very little loss going from \cref{thm:err_int} to the much easier to evaluate \cref{eqn:ex_cif_pw_a_priori2}.
        An a posteriori bound obtained by using \cref{thm:err_int} with \( S_0 = \mathcal{I}(\vec{A}) \) and \( S_i = \{ \lambda_i(\vec{T}_k) \} \) ({\protect\raisebox{0mm}{\protect\includegraphics[scale=.7]{imgs/legend/dash_pink.pdf}}}). Observe that all bounds, especially the a posteriori ones  closely match the true convergence of Lanczos-FA.
    }
    \label{fig:lanczos_CIF_error_bounds}
\end{figure}
Our next example illustrates the application of \cref{thm:err_int} to several common piecewise analytic functions.
Functions of this class have found widespread use throughout scientific computing and data science but have proven particularly difficult to analyze using existing approaches \cite{di_napoli_polizzi_saad_16,frostig_musco_musco_sidford_16,jin_sidford_19,eshof_frommer_lippert_schilling_van_der_vorst_02}.

\begin{example}[Step and Absolute Value Functions]
\label{ex:cif_pw}
Let \( f(x) \) be one of \( |x-a| \), \( \operatorname{step}(x-a) \), or \( \operatorname{step}(x-a)/x \) for \( a \in \mathcal{I}(\vec A) \), where, for \( z \in \mathbb{C} \) we define $\operatorname{step}(z) := 0$ for $\text{Re}(z) < 0$ and $\operatorname{step}(z) := 1$ for $\text{Re}(z) \ge 0$.  Also, for \( z \in \mathbb{C} \), we replace \( |x-a| \) by \( z-a \) if \( \text{Re}(z) > a \) and by \( a-z \) if \( \text{Re}(z) \leq a \). %\Cam{I defined step here. Is this correct/needed?} - yeah, this is right
Note that the latter two functions correspond to principle component projection and principle component regression respectively.
In the case of principle component regression, we $\vec{A}$ is positive semi-definite.
The step function is also closely related to the sign function, which is widely used in quantum chromodynamics to compute the overlap operator \cite{eshof_frommer_lippert_schilling_van_der_vorst_02}.

Next, take \( w = a \) and define \( \Gamma_1 \) and \( \Gamma_2 \) as the boundaries of the disks \( \mathcal{D}_1 := \mathcal{D} ( \lambda_{\text{min}}(\vec{A}), w - \lambda_{\text{min}}(\vec{A}) - \varepsilon) \) and \( \mathcal{D}_2 := \mathcal{D} ( \lambda_{\text{max}}(\vec{A}) , \lambda_{\text{max}}(\vec{A}) - w - \varepsilon ) \), for some sufficiently small \( \varepsilon > 0 \). 
Then \( f \) is analytic in a neighborhood of the union of these two disks, so assuming none of the eigenvalues of $\vec{A}$ or $\vec{T}_k$ are equal to $a$, we can apply \cref{thm:level_sets}.

Note that \( \|h_{w,z}\|_{\mathcal{I}(\vec{A})} \to 1 \) as \( z \to w \) from outside \( [a,b] \), avoiding a potential singularity which would occur if the contour $\Gamma$ passed through $\mathcal{I}(\vec{A})$ at any other points.
In fact, ignoring the contribution of $\epsilon$, \( \| h_{w,z} \|_{\mathcal{I}(\vec{A})} = 1 \) for all \( z\in \Gamma_1 \) and for all \( z\in\Gamma_2 \).
Thus, \cref{thm:disk} can be written as
\begin{align}
    \| f(\vec{A}) \vec{b} - \lan_k(f) \|
%    &\leq \left( \frac{1}{2\pi} \sum_{j=1}^{2} \int_{\Gamma_j} | f(z) | \d{z}| \right) \| \err_k(w) \| \label[ineq]{eqn:ex_cif_pw_a_priori1}
    &\leq \left(\frac{1}{2\pi} \sum_{j=1}^{2}  |\Gamma_j| \max_{z\in\Gamma_j} | f(z) | \right) \| \err_k(w) \| \label[ineq]{eqn:ex_cif_pw_a_priori2}.
\end{align}
The values of this bound for all three functions are summarized in \cref{tab:cif_piecewise}.

In \cref{fig:lanczos_CIF_error_bounds}, we plot the bounds from \cref{thm:err_int} for the contour described above with \( S_i = \mathcal{I}(\vec{A}) \).
%For reference, we also plot the bound \cref{eqn:ex_cif_pw_a_priori2} which shows that there is not much loss from using \cref{eqn:ex_cif_pw_a_priori2} in place of \cref{thm:err_a_priori}.
%Again, all integrals are computed using SciPy's \texttt{integrate.quad}.

\begin{table}[htb]
 \begin{align*}
\begin{array}{rlll}
    \toprule
    f(x) & f(z), z \in \Omega_1 & f(z), z \in \Omega_2 & \frac{1}{2 \pi} \sum_{j=1}^{2} | \Gamma_j | \max_{z\in\Gamma_j} | f(z) | \\\midrule
    |x-a| & a-z & z-a &  2(a- \lambda_{\text{min}} )^2 + 2(\lambda_{\text{max}}-a)^2  \\
    \operatorname{step}(x-a) & 0 & 1 &  (\lambda_{\text{max}}-a) \\
    \operatorname{step}(x-a)/x & 0 & 1/z & (\lambda_{\text{max}}-a)/a\\
    \bottomrule
\end{array}
\end{align*}
\caption{Values of the factor in parentheses on the right-hand side of \cref{eqn:ex_cif_pw_a_priori2} (ignoring $\varepsilon$) for several common piecewise analytic functions.}
\label{tab:cif_piecewise}
\end{table}

If \( w \in \mathcal{I}(\vec{A}) \)  we note that \( \| \err_k(w) \| \) corresponds to the indefinite linear system \( (\vec{A} - w \vec{I}) \vec{x} = \vec{b} \), so standard results for the Conjugate Gradient algorithm are not applicable.
However, the residual of this system can still be computed exactly once the Lanczos factorization \cref{eqn:lanczos_factorization} has been obtained, and as we prove in \cref{sec:linear_systems}, a priori bounds for the convergence of MINRES \cite{cullum_greenbaum_96} can be extended to the Lanczos algorithm for indefinite systems.
It is also clear that, at the cost of having to compare against the error of multiple different linear systems, functions which are piecewise analytic on more than two regions can be handled.

\end{example}

\section{Finite precision}
\label{sec:finite_precision}

While reorthogonalization  in the Lanczos method (\cref{alg:lanczos}) is unnecessary in exact arithmetic, omitting it may result in drastically different behavior when using finite precision arithmetic; see for instance \cite{meurant_strakos_06}.
In the context of Lanczos-FA, the two primary effects are (i) a delay of convergence (increase in the number of iterations to reach a given level of accuracy) and (ii) a reduction in the maximal attainable accuracy.
These effects are reasonably well understood in the context of linear systems \cite{greenbaum_89,greenbaum_97a}, i.e., \( f(x) = 1/x \), and for some other functions such as the matrix exponential \cite{druskin_greenbaum_knizhnerman_98}. However, general theory is  limited. A notable exception is \cite{musco_musco_sidford_18}, which argues that the uniform error bound for Lanczos-FA \cref{eqn:poly_unif} holds to a close degree in finite precision arithmetic.

When run without reorthogonlization, \cref{alg:lanczos} will produce \( \vec{Q}_k \) and \( \vec{T}_k \) satisfying a perturbed three term recurrence
\begin{align}
    \label{eqn:lanczos_factorization_fp}
    \vec{A} \vec{Q}_k = \vec{Q}_k \vec{T}_k + \beta_k \vec{q}_{k+1} \vec{e}_k^{\T} + \vec{F}_k,
\end{align}
where \( \vec{F}_k \) is a perturbation term.
Moreover, the columns of \( \vec{Q}_k \) may no longer be orthogonal.
A priori bounds on the size of \( \vec{F}_k \) and the loss of orthogonality between successive Lanczos vectors have been established in a series of works by Paige \cite{paige_71,paige_76,paige_80,paige_19}.
These quantities can also be computed easily once \( \vec{Q}_k \) and \( \vec{T}_k \) have been obtained, allowing for easy use with our bounds.
\iffalse
We remark that the full reorthogonalization can be used for a cost of \( O(nk^2) \).
This will ensure \( \vec{Q}_k \) and \( \vec{T}_k \) match the exact arithmetic outputs to near machine precision. \Anne{I don't think you can claim this.  Perhaps just omit this sentence.}
It is also possible to selectively orthogonalize only against converged Ritz vectors \cite{parlett_scott_79}, which, based on Paige's analysis, is sufficient to mitigate the effects of roundff error.
This has the benefit of avoiding the cost of full reorthogonalization.
\fi

\subsection{Effects of finite precision on our error bounds for Lanczos-FA}

Note that using the divide and conquer algorithm from \cite{gu_eisenstat_95} to compute the eigendecomposition of the tridiagonal matrix \( \vec{T}_k \), we can quickly and stably compute \( \vec{Q}_k f(\vec{T}_k) \vec{e}_1 \).
A detailed analysis of this is given in \cite[Appendix A]{musco_musco_sidford_18}.

While the tridiagonal matrix \( \vec{T}_k \) and the matrix \( \vec{Q}_k \) of Lanczos vectors produced in finite precision arithmetic may be very different from those produced in exact arithmetic, we now show that our error bounds, based on the \( \vec{T}_k \) and \( \vec{Q}_k \) actually produced, still hold to a close approximation.
%As we now show, the primary effect of finite precision on our bounds is due to the presense of the perturbation term \( \vec{F}_k \) in \cref{eqn:lanczos_factorization_fp}.
First, we argue that \cref{thm:shifted_lanczos_equivalence} holds to a close degree provided \( \vec{F}_k \) is not too large.
Towards this end, note that we have the shifted perturbed recurrence,
\begin{align}
    \label{eqn:shifted_lanczos_factorization_fp}
    ( \vec{A} - z \vec{I} ) \vec{Q}_k
    &= \vec{Q}_k ( \vec{T}_k - z \vec{I} ) + \beta_k \vec{q}_{k+1} \vec{e}_k^{\T} + \vec{F}_k.
\end{align}

From \cref{eqn:shifted_lanczos_factorization_fp}, it is then clear that,
    \begin{align*}
        %(\vec{A} - z\vec{I}) \lan_k ( h_z(x) )
        (\vec{A}-z \vec{I}) \vec{Q}_k (\vec{T}_k - z \vec{I})^{-1} \vec{e}_1
        &= \vec{Q}_k \vec{e}_1 + \beta_k \vec{q}_{k+1} \vec{e}_k^{\T} (\vec{T}_k - z \vec{I})^{-1} \vec{e}_1 + \vec{F}_k (\vec{T}_k - z \vec{I})^{-1} \vec{e}_1.
    \end{align*}

This implies that \cref{thm:shifted_linear_system_error} also holds closely.
More specifically,
\begin{align*}
    \Res_k(z)
    &= \det(h_{w,z}(\vec{T}_k)) \Res_k(w)
    + \vec{f}_k(w,z)
\\
    \err_k(z)
    &= \det(h_{w,z}(\vec{T}_k))  h_{w,z}(\vec{A}) 
    \err_k(w)
    + (\vec{A}-z\vec{I})^{-1} \vec{f}_k(w,z)
\end{align*}
where
\begin{align*}
    \vec{f}_k(w,z) := 
    \vec{F}_k \left( (\vec{T}_k - z\vec{I})^{-1} - \det(h_{w,z}(\vec{T}_k)) (\vec{T}_k-w\vec{I)}^{-1} \right)\vec{e}_1.
\end{align*}

Using this we have,
\begin{align*}
    f(\vec{A})\vec{b} - \lan_k(f)
    &= - \frac{1}{2\pi i} \oint_{\Gamma} f(z) \err_k(z) \d{z} - \frac{1}{2\pi i} \oint_{\Gamma} f(z) (\vec{A} - z\vec{I})^{-1} \vec{f}_k(w,z) \d{z}
\end{align*}
which we may bound using the triangle inequality as
\begin{align*}
    \| f(\vec{A})\vec{b} - \lan_k(f) \|
    &\leq \frac{1}{2\pi} \left\| \oint_{\Gamma} f(z) \err_k(z) \d{z} \right \| +  \frac{1}{2\pi} \left\|  \oint_{\Gamma} f(z) (\vec{A} - z\vec{I})^{-1} \vec{f}_k(w,z) \d{z} \right\|.
\end{align*}
This expression differs from \cref{thm:err_int} only by the presence of the term involving \( \vec{f}_k(w,z) \) (and, of course, by the fact that \( \err_k(z) \) now denotes the error in the finite precision computation).
If we take \( \| \cdot \| \) as the \( (\vec{A}-w\vec{I})^2 \)-norm, then this additional term can be bounded by,
\begin{align}
    \hspace{4em}&\hspace{-4em} \nonumber
    \frac{1}{2\pi} \left\|  \oint_{\Gamma} f(z) (\vec{A} - z\vec{I})^{-1} \vec{f}_k(w,z) \d{z} \right\|
    \\&\leq \frac{1}{2\pi} \oint_{\Gamma} |f(z)|  \cdot \| (\vec{A}-w\vec{I})(\vec{A} - z\vec{I})^{-1} \|_2 \cdot\| \vec{f}_k(w,z) \|_2  \cdot |\d{z}| \nonumber
    \\&\leq \frac{1}{2\pi} \oint_{\Gamma} |f(z)| \cdot \|h_{w,z}\|_{S_0} \cdot \| \vec{f}_k(w,z) \|_2 \cdot |\d{z}|. \label[ineq]{eqn:a_posteriori_fp}
\end{align}

Note that \cref{eqn:a_posteriori_fp} can be viewed as an upper bound of the ultimate obtainable  accuracy of Lanczos-FA in finite precision after convergence.
 If the inequalities do not introduce too much slack, this upper bound will also produce a reasonable estimate.
If \( \| \vec{F}_k \| \) is small, the size of this addition is also hopefully small, in which case one may simply ignore the contribution of \cref{eqn:a_posteriori_fp}, provided the Lanczos-FA error is not near the final accuracy.
We have worked in the \( (\vec{A}-w\vec{I})^2 \) norm as it simplifies some of the analysis, but in principle, a similar approach could be used with other norms.
This is straightforward, but would involve bounding something other than \( \|h_{w,z}\|_{S_0} \).

\begin{example}
    The left panel of \cref{fig:sqrt_fp} shows the convergence of Lanczos-FA when \cref{alg:lanczos}, \emph{without reorthogonalization}, is used to generate \( \vec{Q}_k \) and \( \vec{T}_k \).
Compared with the error of the iterates generated using full orthogonalization, a delay of convergence and loss of accuracy are clear.
This figure also shows the error bounds derived by bounding \( \|\vec{F}_k\| \) as described above. 
We note that the contribution from the integral in \cref{eqn:a_posteriori_fp} is almost negligible until the bound is near the final accuracy.

\begin{figure}[ht]
    \begin{subfigure}[t]{.48\textwidth}\centering
        \includegraphics[width=\textwidth]{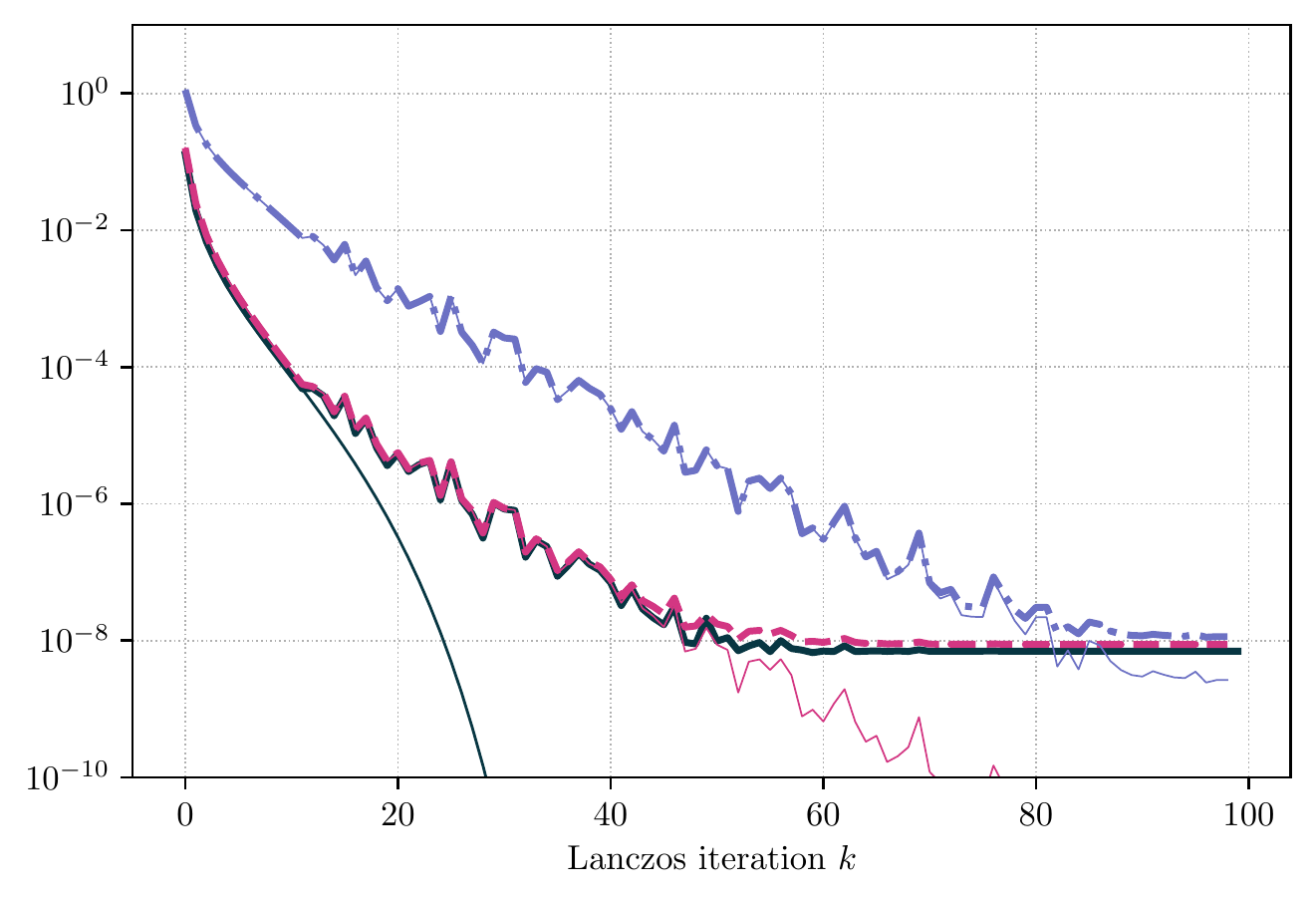}
        \caption{
        \( ( \vec{A} - w\vec{I})^2 \)-norm error bounds for \( f(x) = \sqrt{x} \) using a Pac-Man contour ($r\to0$, $R\to\infty$). \emph{Legend}: True Lanczos-FA error ({\protect\raisebox{0mm}{\protect\includegraphics[scale=.7]{imgs/legend/solid_blue.pdf}}}).
        A priori bounds obtained by using \cref{thm:err_int} with \( S_0 = S_i = \mathcal{I}(\vec{A}) \) with
        ({\protect\raisebox{0mm}{\protect\includegraphics[scale=.7]{imgs/legend/dashdot_purple.pdf}}})
        and without
        ({\protect\raisebox{0mm}{\protect\includegraphics[scale=.7]{imgs/legend/thin_purple.pdf}}})
        right hand side of \cref{eqn:a_posteriori_fp}.
        A posteriori bounds obtained by using \cref{thm:err_int} with \( S_0 = \mathcal{I}(\vec{A}) \) and \( S_i = \{ \lambda_i(\vec{T}_k) \} \) with ({\protect\raisebox{0mm}{\protect\includegraphics[scale=.7]{imgs/legend/dash_pink.pdf}}})
        and without
        ({\protect\raisebox{0mm}{\protect\includegraphics[scale=.7]{imgs/legend/thin_pink.pdf}}})
        right hand side of \cref{eqn:a_posteriori_fp}.
        For reference, the convergence of Lanczos-FA with reorthogonalization in double precision ({\protect\raisebox{0mm}{\protect\includegraphics[scale=.7]{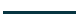}}}) is also shown.
        }
    \end{subfigure}\hfill 
    \begin{subfigure}[t]{.48\textwidth}\centering
        \includegraphics[width=\textwidth]{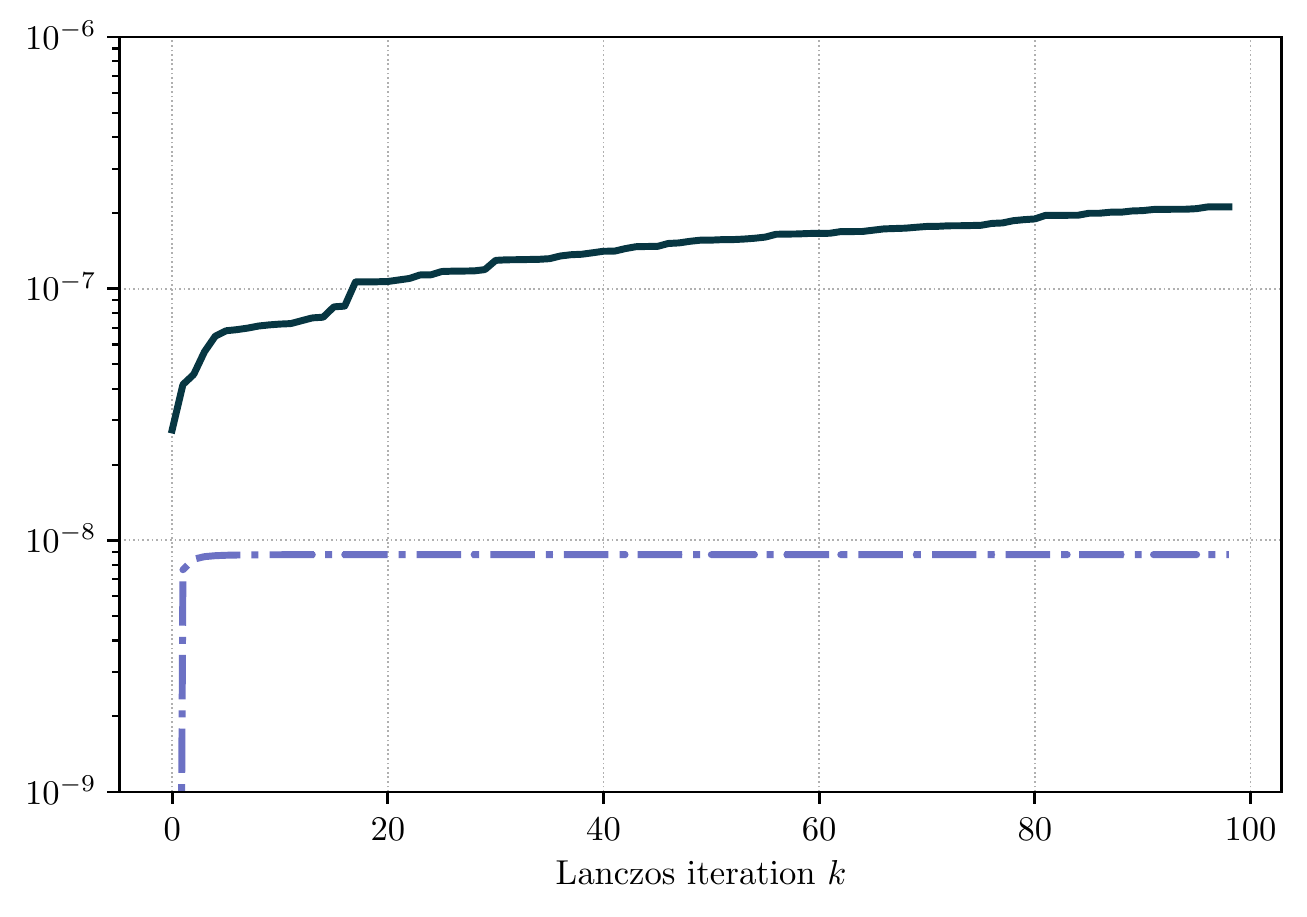}
        \caption{
        \emph{Legend}:
        \( \| \vec{F}_k \|_{\mathsf{F}}  \) ({\protect\raisebox{0mm}{\protect\includegraphics[scale=.7]{imgs/legend/solid_blue.pdf}}}), 
        %\( \| [\vec{F}_k]_{:,k} \|_2 \) %({\protect\raisebox{0mm}{\protect\includegraphics[scale=.7]{imgs/legend/dash_pink.pdf}}},
        right hand side of \cref{eqn:a_posteriori_fp} 
        ({\protect\raisebox{0mm}{\protect\includegraphics[scale=.7]{imgs/legend/dashdot_purple.pdf}}}).
        Note that the size of \( \vec{F}_k \) is small relative to the Lanczos-FA error, until the accuracy is near the final accuracy. 
        }
    \end{subfigure}
    \caption[]{
        \( \vec{A} \) has \( n=50 \) eigenvalues with \( \lambda_1 = 1 \), \( \lambda_n = 0.001 \), and \( \lambda_i = \lambda_n + \frac{n-i}{n-1} (\lambda_1 - \lambda_n ) \rho^{i-1} \), \(i=2, \ldots , n-1 \), as described in \cite{strakos_greenbaum_92} with parameter \( \rho = 0.8 \). 
        Here Lanczos is run without reorthogonalization in single precision arithmetic, but the integrals are evaluated using double precision arithmetic.
    }
    \label{fig:sqrt_fp}
\end{figure}
\end{example}

\section{Quadratic forms}
\label{sec:quadratic_form}

In many applications, one seeks to compute \( \vec{b}^\cT f(\vec{A}) \vec{b} \) rather than \( f(\vec{A})\vec{b} \).
A common approach is Lanczos quadrature, which computes the approximation \( \vec{b}^\cT \lan_k(f) \) to \( \vec{b}^\cT f(\vec{A}) \vec{b} \).
This approximation is a degree \( k \) Gaussian quadrature approximation to the integral of \( f \) against the weighted spectral measure corresponding to \( \vec{A}, \vec{b} \); see for instance \cite{chen_trogdon_ubaru_21,golub_meurant_09,ubaru_chen_saad_17}.
However, as with the case of Lanczos-FA, most existing error bounds for Lanczos quadrature are either pessimistic or limited to special classes of functions.
%Some attempts to provide more refined abounds have been made for the casegenerated of spectrum approximation \cite{chen_trogdon_ubaru_21}

Note that the Lanczos-FA approximation satisfies,
\begin{align*}
    \vec{b}^\cT \lan_k(f) 
    = \vec{b}^\cT \vec{Q}_k f(\vec{T}_k) \vec{Q}_k^\cT \vec{b}
    = \| \vec{b} \|_2^2\: \vec{e}_1^\cT f(\vec{T}_k) \vec{e}_1.
\end{align*}
Thus, we can compute \( \vec{b}^\cT \lan_k(f) \) \emph{without} storing or recomputing \( \vec{Q}_k \).

Since \( \vec{A} \) is Hermitian, \( (\vec{A} - z\vec{I})^\cT = \vec{A} - \overline{z} \vec{I} \). 
Thus, since 
\begin{align*}
    \vec{b}^\cT (\vec{A} - z\vec{I})^{-1}
    = ( ( \vec{A} - \overline{z} \vec{I}   )^{-1} \vec{b} )^\cT
    = ( \lan_k(h_{\overline{z}}) + \err_k(\overline{z}) ) \vec{b} )^\cT
\end{align*}
we can expand the quadratic form error as
\begin{align*}
    \vec{b}^\cT \err_k(z)
    = \vec{b}^\cT ( \vec{A} - z \vec{I})^{-1} \Res_k(z)
    = \left( \lan_k ( h_{\overline{z}} ) ) + \err_k(\overline{z}) \right)^\cT \Res_k(z).
\end{align*}
Now, by definition, \( \lan_k ( h_{\overline{z}}(x) )  = \vec{Q}_k h_{\overline{z}}(\vec{T}_k) \vec{Q}^\cT \vec{b} \) and by \cref{thm:shifted_lanczos_equivalence} \( \Res_k(z) \) is proportional to \( \vec{q}_{k+1} \).
Thus, since, at least in exact arithmetic, \( \vec{q}_{k+1} \) is orthogonal to $\vec{Q}_k$,
\begin{align*}
    \vec{b}^\cT \err_k(z)
    &= \err_k(\overline{z})^\cT \Res_k(z)
    = (( \vec{A} - \overline{z}\vec{I})^{-1}  \Res_k(\overline{z}))^\cT \Res_k(z).
\end{align*}
Next, using \cref{thm:shifted_linear_system_error} and the fact that \( h_{w,z}(x) h_{w,\overline{z}}(x) = |h_{w,z}(x)|^2 \) for \( w,x\in\R \), 
\begin{align*}
    \vec{b}^\cT \err_k(z)
    &= |\det(h_{w,z}(\vec{T}_k))|^2 \Res_k(w)^\cT ( \vec{A} - z\vec{I})^{-1} \Res_k(w).
\end{align*}

We then have,
\begin{align*}
    | \vec{b}^\cT \err_k(z) |
    \leq | \! \det(h_{w,z}(\vec{T}_k)) |^2 \cdot \| ( \vec{A} - z\vec{I})^{-1} \|_2 \cdot \| \Res_k(w) \|_2^2.
\end{align*}
\iffalse
which gives
\begin{align*}
    | \vec{b}^\cT f(\vec{A}) \vec{b} - \vec{b}^\cT \lan_k(f) |
    &\leq \left( \frac{1}{2\pi} \oint_{\Gamma} | f(z) | \cdot |\det(h_{w,z}(\vec{T}_k))|^2 \cdot \| h_z(\vec{A}) \|_2 \cdot |\d{z}| \right) \| \Res_k(w) \|_2^2.
\end{align*}
\fi
Applying the Cauchy integral formula we therefore obtain a bound for the quadratic form error analogous to \cref{thm:err_int},
\begin{align}
     | \vec{b}^\cT f(\vec{A}) \vec{b} - \vec{b}^\cT \lan_k(f) |
    & \leq  \left( \frac{1}{2\pi} \oint_{\Gamma} | f(z) |  \cdot \left( \prod_{i=1}^{k} \|h_{w,z}\|_{S_i}^2  \right) \!\cdot \| h_z \|_{S_0}  \!\cdot |\d{z}| \right) \| \Res_k(w) \|_2^2.  \label{eqn:err_quad_int}
\end{align}
Comparing the above to the bound of \cref{thm:err_int} for approximating $f(\vec{A})\vec{b}$, we see that $\|\err_k(w)\|$ is replaced with $\|\Res_k(w)\|_2^2$. 
Thus, heuristically, we can expect the quadratic form to converge at a rate twice that of the norm of the error of the matrix function.

Similar to \cref{thm:Q_wz_value} we have the following bound on $\|h_z\|_{S_i}$ when $S_0$ is an  interval.
This allows a bound on \cref{eqn:err_quad_int} analogous to \cref{eqn:integral_error}.
\begin{lemma}
\label{thm:Qz}
For any interval $[a,b]\subset \R$, if $z\in\mathbb{C}\setminus[a,b]$, we have    
    %Fix \( z \in \mathbb{C} \setminus [a,b] \) and \( w\in\mathbb{R} %\).
    \begin{align*}
        \| h_z \|_{[a,b]}
    =
    \begin{cases}
        1/|\Im(z)| &  \Re(z) \in \mathcal{I}(\vec{A}) \\
        1/|a-z| & \Re(z) < a \\
        1/|b-z| & \Re(z) > b
    \end{cases}
\end{align*}   
\end{lemma}

\begin{figure}[ht]
        \begin{subfigure}[t]{.48\textwidth}\centering
        \includegraphics[width=\textwidth]{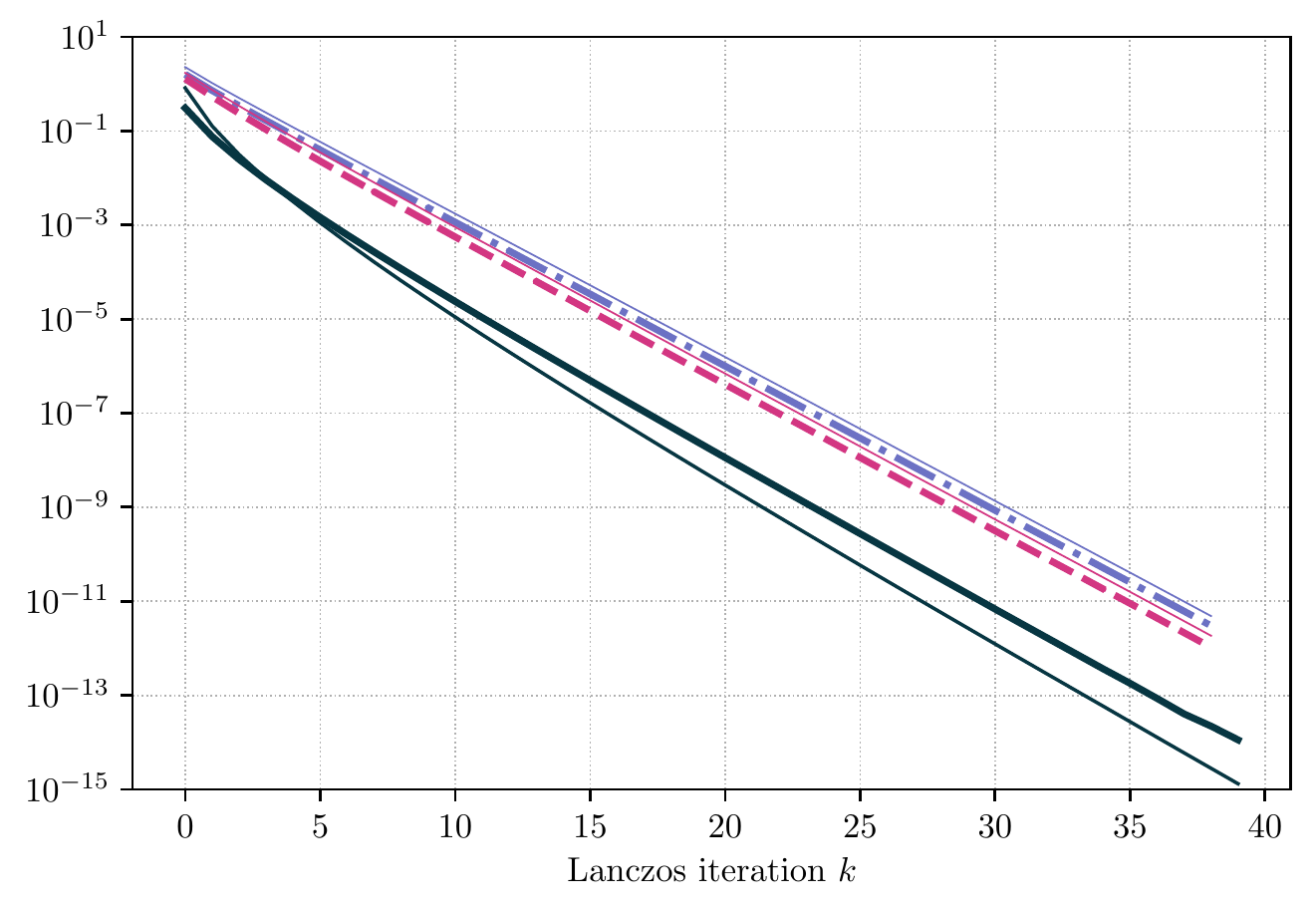}
        \caption{Error bounds for \( f(x) = \log(x) \) using a Pac-Man contour as described in \cref{ex:quadform_log}. \( \vec{A} = \vec{X}\vec{X}^\T \), where the entries of \( \vec{X} \in\mathbb{R}^{n,2n} \) are independent Gaussians with mean zero variance \( 1/2n \) where \( n = 3000 \).
        \emph{Legend}: 
        A priori bounds obtained by using \cref{eqn:err_quad_int} with \( S_0 = S_i = \mathcal{I}(\vec{A})  \) ({\protect\raisebox{0mm}{\protect\includegraphics[scale=.7]{imgs/legend/dashdot_purple.pdf}}}) and 
        \( S_0 = S_i = \tilde{\mathcal{I}}(\vec{A}) \)
        ({\protect\raisebox{0mm}{\protect\includegraphics[scale=.7]{imgs/legend/thin_purple.pdf}}}).
        A posteriori bounds obtained by using \cref{eqn:err_quad_int} with \( S_0 = \mathcal{I}(\vec{A}) \), \( S_i = \{ \lambda_i(\vec{T}_k) \} \) ({\protect\raisebox{0mm}{\protect\includegraphics[scale=.7]{imgs/legend/dash_pink.pdf}}}) and 
        \( S_0 = \tilde{\mathcal{I}}(\vec{A}) \), \( S_i = \{ \lambda_i(\vec{T}_k) \} \)
        ({\protect\raisebox{0mm}{\protect\includegraphics[scale=.7]{imgs/legend/thin_pink.pdf}}}).
        %Note that because this matrix has essentially continuous eigenvalue density (following the Marchenko–Pastur distribution) relative to the number of Lanczos steps, the rate of convergence more or less matches the asymptotic bounds.
    }
    \end{subfigure}\hfill
    \begin{subfigure}[t]{.48\textwidth}\centering
        \includegraphics[width=\textwidth]{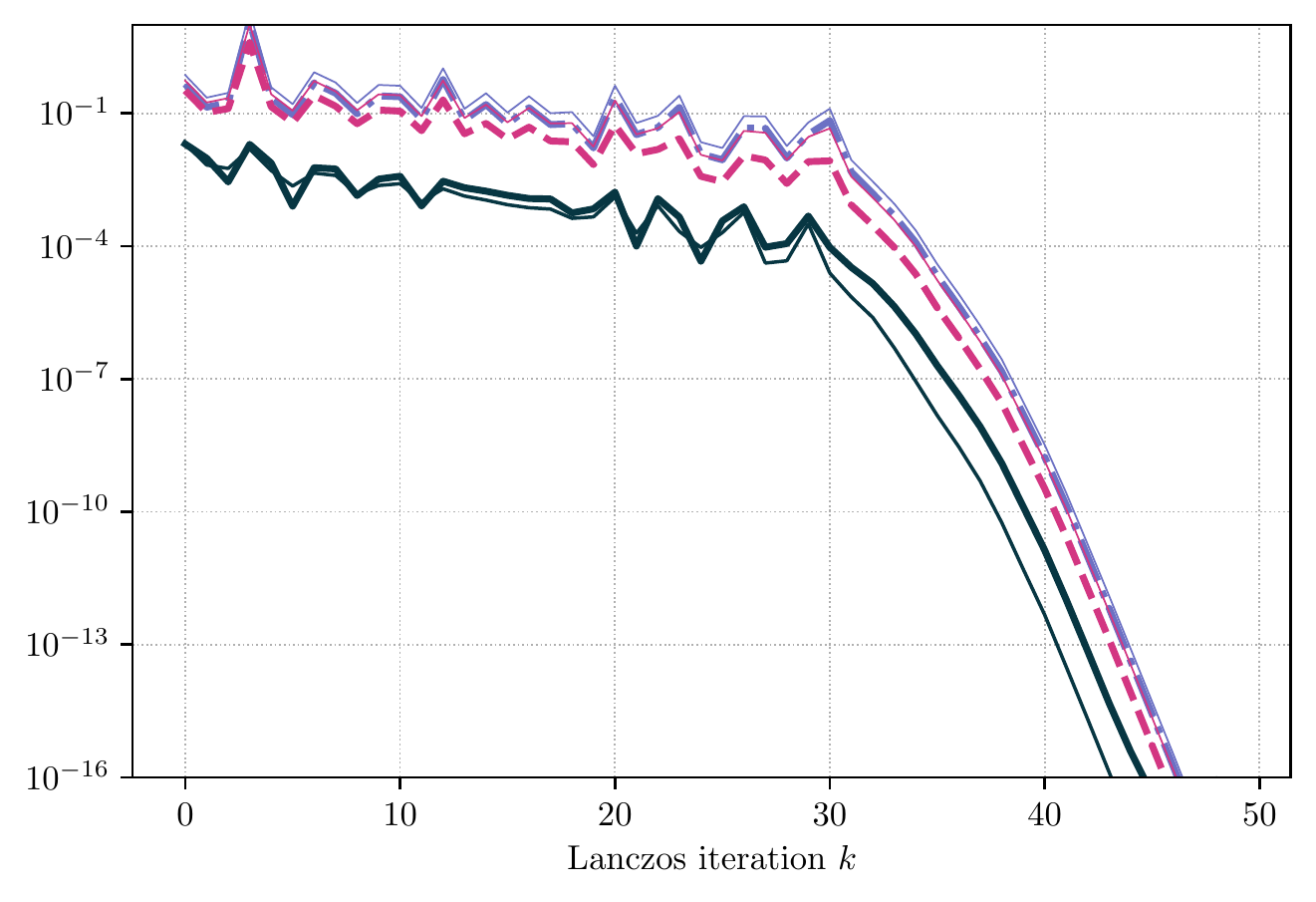}
        \caption{Error bounds for \( f(x) = \operatorname{step}(x-a) \) with double circle contour as described in \cref{ex:quadform_step}. \( \vec{A} \) is the covariance matrix of the MNIST training data \cite{lecun_cortes_burges_10}.
        \emph{Legend}:
        A priori bounds obtained by using \cref{eqn:err_quad_int} with \( S_0 = S_i = \mathcal{I}_w(\vec{A})  \) ({\protect\raisebox{0mm}{\protect\includegraphics[scale=.7]{imgs/legend/dashdot_purple.pdf}}}) and 
        \( S_0 = S_i = \tilde{\mathcal{I}}_w(\vec{A}) \)
        ({\protect\raisebox{0mm}{\protect\includegraphics[scale=.7]{imgs/legend/thin_purple.pdf}}}).
        A posteriori bounds obtained by using \cref{eqn:err_quad_int} with \( S_0 = \mathcal{I}_w(\vec{A}) \), \( S_i = \{ \lambda_i(\vec{T}_k) \} \) ({\protect\raisebox{0mm}{\protect\includegraphics[scale=.7]{imgs/legend/dash_pink.pdf}}}) and 
        \( S_0 = \tilde{\mathcal{I}}_w(\vec{A}) \), \( S_i = \{ \lambda_i(\vec{T}_k) \} \)
        ({\protect\raisebox{0mm}{\protect\includegraphics[scale=.7]{imgs/legend/thin_pink.pdf}}}).
        }
    \end{subfigure} 
    \caption{
        Lanczos-FA quadratic form errors.
        \emph{Legend}:
        \( | \vec{b}^\cT f(\vec{A}) \vec{b} - \vec{b}^\cT \lan_k(f) | \) ({\protect\raisebox{0mm}{\protect\includegraphics[scale=.7]{imgs/legend/solid_blue.pdf}}}).
        For reference we also show \( \| f(\vec{A})\vec{b} - \lan_k(f) \|_2^2 \)
        ({\protect\raisebox{0mm}{\protect\includegraphics[scale=.7]{imgs/legend/thin_blue.pdf}}}). 
        Note that this is the square of the 2-norm of the Lanczos-FA error.
    }
    \label{fig:quad_form}
\end{figure}

In the case that the contour $\Gamma$ does not pass through \( \mathcal{I}(\vec{A}) \), the bound of \cref{eqn:err_quad_int} is essentially  as easy to compute as that of  \cref{thm:err_int}.
However, if the contour passes through \( \mathcal{I}(\vec{A}) \) at \( w \), to ensure that \( S_0 \) does not contain points in the contour, it must be chosen as a set other than $\mathcal{I}(\vec A)$. This set must contain all of $\vec A$'s eigenvalues and we must bound its distance to the contour (in particular, to $w$).

\begin{example}
\label{ex:quadform_log}
Suppose \( \vec{A} \) is positive definite and \( f(x) = \log(x) \).
We use \cref{eqn:err_quad_int} to obtain a bound for the quadratic form error $| \vec{b}^\cT f(\vec{A})\vec{b} - \vec{b}^\cT \lan_k(f)|$.
A priori bounds are obtained with $S_0, S_i = \mathcal{I}(\vec A)$ while a posteriori bounds are obtained with $S_0 = \mathcal{I}(\vec A)$ and $S_i = \{ \lambda_i(\vec{T}_k) \}$.
In both cases, we take $\Gamma$ as the Pac-Man contour centered at 0 with $r=\lambda_{\text{min}}(\vec{A})/100$ to avoid the singularity $\log(0) = -\infty$.
The resulting bounds are shown in the left panel of \cref{fig:quad_form}.

As in \cref{ex:sqrt_contours}, we also consider the cases where we use an estimate $\tilde{\mathcal{I}}(\vec{A})$ for $\mathcal{I}(\vec{A})$ to study the sensitivity of our bounds to $S_i$.
For these tests we use a Pac-Man contour with $r=\lambda_{\text{min}}(\vec{A})/200$.
\end{example}

\begin{example}
\label{ex:quadform_step}
Let \( f(x) = \operatorname{step}(x-a) \) for  \( a \in \mathcal{I}(\vec{A}) \), and set $w = a$.
Similarly to the previous example we use \cref{eqn:err_quad_int} to obtain a bound for the quadratic form error $| \vec{b}^\cT f(\vec{A})\vec{b} - \vec{b}^\cT \lan_k(f)|$. 
However, we must have $S_i$ avoid where $\Gamma$ crosses the real axis.

Suppose \( \lambda_{\text{max}}^{\text{l},w}(\vec{A}) \) and \( \lambda_{\text{min}}^{\text{r},w}(\vec{A}) \) are consecutive eigenvalues of \( \vec{A} \) so that  \( \lambda_{\text{max}}^{\text{l},w}(\vec{A}) < w < \lambda_{\text{min}}^{\text{r},w}(\vec{A}) \).
Then we can define
\begin{align*}
\mathcal{I}_w(\vec{A}) := [ \lambda_{\text{min}}(\vec{A}) , \lambda_{\text{max}}^{\text{l},w}(\vec{A}) ] \cup [ \lambda_{\text{min}}^{\text{r},w}(\vec{A}) , \lambda_{\text{max}}(\vec{A}) ].
\end{align*}
In this case, \( \|h_{z}\|_{\mathcal{I}_w(\vec{A})} = \max\{ \|h_{z}\|_{ [\lambda_{\text{min}},\lambda_{\text{max}}^{\text{l},w} ] }, \|h_{z}\|_{ [\lambda_{\text{min}}^{\text{r},w},\lambda_{\text{max}} ] }) \} \) can be computed using \cref{thm:Qz}. 

We can then apply \cref{eqn:err_quad_int} to obtain a bound for the quadratic form error $| \vec{b}^\cT f(\vec{A})\vec{b} - \vec{b}^\cT \lan_k(f)|$.
A priori bounds are obtained with $S_0, S_i = \mathcal{I}_w(\vec A)$ while a posteriori bounds are obtained with $S_0 = \mathcal{I}_w(\vec A)$ and $S_i = \{ \lambda_i(\vec{T}_k) \}$.
This is shown in the right panel of \cref{fig:quad_form}.
Of course, in practice it is unlikely that \( \lambda_{\text{min}}^{\text{l},w}(\vec{A}) \) and \( \lambda_{\text{max}}^{\text{r},w}(\vec{A}) \) are known.
The distance to \( w \) of course can be estimated by estimating the smallest eigenvalue of \( (\vec{A} - w\vec{I})^2 \), perhaps via Lanczos.
However, it can be expected to be more difficult than estimating \( \lambda_{\text{min}}(\vec{A}) \) and \( \lambda_{\text{max}}(\vec{A}) \).
Thus, we also show the effect of approximating \( \lambda_{\text{max}}^{\text{l},w}(\vec{A}) \) and \( \lambda_{\text{min}}^{\text{l},w}(\vec{A}) \).
Specifically, we compute \( \| h_{w,z} \|_{\tilde{\mathcal{I}}_w(\vec{A})} \) where 
\begin{align*}
    \tilde{\mathcal{I}}_w(\vec{A}) = [\lambda_{\text{min}}/2,w-\gamma]\cup[w+\gamma,1.5\lambda_{\max}]) 
\end{align*}
for \( \gamma = \min_{\lambda\in\Lambda(\vec{A})} | \lambda - w | / 100 \).
\end{example}

\section{Conclusion and outlook}

In this paper we give a simple approach to generate error bounds for Lanczos-FA used to approximate \( f(\vec{A})\vec{b} \) when \( f(x) \) is piecewise analytic.
Our framework can be used both a priori and a posteriori, and the bounds, to close degree, hold in finite precision.
While outside the scope of this paper, the same general approach is applicable to non-Hermitian matrices computed using an Arnoldi factorization.

\section{Acknowledgments}
The authors thank Thomas Trogdon for suggestions in early stages.

\appendix

\section{Error bounds for Lanczos on linear systems}
\label{sec:linear_systems}
Our analysis reduces understanding the Lanczos-FA error for a function $f$ to understanding \( \| \err_k(w) \| \), the error of Lanczos-FA used to solve the system \( (\vec{A} - w\vec{I}) \vec{x} = \vec{b} \).
We review several bounds for this task. Without loss of generality, we assume $w=0$, as the $w\vec{I}$ term can be incorperated directly into \( \vec{A}\).

In the case that \( \vec{A}\) is positive (or negative) definite, Lanczos-FA with \( f(x) = 1/x  \) is equivalent to the conjugate gradient algorithm (CG) \cite{hestenes_stiefel_52}.
Therefore, it inherits CG's well known property of returning an optimal solution in the $\vec{A}$-norm (or $-\vec{A}$-norm if $\vec{A}$ is negative definite).
% If \( \vec{A} \) is positive definite, it is well known that the Lanczos-FA approximation is optimal in the \( \vec{A} \) norm for \( f(x) = 1/x \).
That is,
\begin{align*}
    \| \err_k \|_{\vec{A}}
    = \min_{\vec{y} \in \mathcal{K}_k(\vec{A},\vec{b}) } \| \vec{A}^{-1} \vec{b} - \vec{y} \|_{\vec{A}}
    = \min_{\substack{\deg p \leq k\\p(0) =1}} \| p(\vec{A}) \vec{A}^{-1} \vec{b} \|_{\vec{A}}.
\end{align*}

From this optimality, we obtain the following (well known) bounds for positive definite \( \vec{A} \)
% Let \( \vec{A} \) be positive definite. 
% Then,
\begin{align*}
    \frac{ \| \err_k \|_{\vec{A}} }{ \| \err_0 \|_{\vec{A}} }
    \leq \min_{\substack{\deg p \leq k\\p(0) =1}} \max_{\lambda\in\Lambda(\vec{A})} | p(\lambda) |
    \leq 2 \left( \frac{\sqrt{\kappa(\vec{A})} -1 }{\sqrt{\kappa(\vec{A})}+1} \right)^k
    \leq 2 \exp\left( - \frac{2k}{ \sqrt{\kappa(\vec{A})} } \right)
\end{align*}
where the final bound follows from the fact that $(x-1)/(x+1) \leq \exp(-2/x)$ for all $x\geq 1$.
The minimax bound, based on the eigenvalues of \( \vec{A} \), is tight in the sense that for each \( k \) there exists \( \vec{b} \) (dependent on \( \vec{A} \) and \( k \)) so that \( \lan_k(f,\vec{A},\vec{b}) \) attains the bound \cite{greenbaum_79}.
The final inequality implies that Lanczos-FA requires \( k \leq \frac{1}{2} \sqrt{\kappa(\vec{A})} \log(2/\epsilon) \) iterations to ensure \( \| \err_k \|_{\vec{A}} / \| \err_0 \|_{\vec{A}} \leq \epsilon \).

From the result above, it is also straightforward to derive a bound that is more directly comparable to \cref{eqn:triangle_ineq1} and \cref{eqn:poly_unif}. Specifically, for $f(x) = 1/x$, \cite{musco_musco_sidford_18} shows:

\begin{align*}
\|\err_k\|_2 = \|f(\vec{A}) \vec{b} - \lan_k(f) \|_2 &\leq \sqrt{\kappa(\vec{A})} \| \vec{b} \|_2\cdot \min_{\deg p< k} \| f-p \|_{\Lambda(\vec{A})}.   \label[ineq]{eqn:a_only}
\end{align*}
Beside the leading constant $\sqrt{\kappa(\vec{A})}$, this bound is strictly stronger than   \cref{eqn:triangle_ineq1} because it only depends on the eigenvalues of $\vec{A}$, and not those of $\vec{T}_k$. As a result, it is also strictly stronger than the uniform approximation bound of \cref{eqn:poly_unif}. 

If \( \vec{A} \) is \emph{indefinite}, we can obtain error bounds by relating the Lanczos-FA approximation to MINRES.
For these bounds, we need the following theorem from \cite{cullum_greenbaum_96} which compares the 2-norm of the residual in the Lanczos approximation to the solution of a Hermitian linear system to that of the MINRES algorithm.
MINRES, by definition, minimizes the 2-norm of the residual over all approximations from the Krylov subspace.

\begin{theorem}
    %\cite[Ex. 5.1]{greenbaum_97}
\label{thm:minres_CG_residuals1}
Let \( \vec{A} \) be a nonsingular Hermitian matrix and define \( \vec{r}_k^M \) as the MINRES residual at step $k$; i.e. 
\begin{align*}
    \vec{r}_k^M := \vec{b} - \vec{A} \hat{\vec{y}}
    ,&&
    \hat{\vec{y}} = \argmin_{\vec{y}\in\mathcal{K}_k(\vec{A},\vec{b})} \| \vec{b} - \vec{A} \vec{y} \|_2 .
\end{align*}
Then, assuming that the initial residuals in the two procedures are the same,
\begin{align*}
    \frac{\| \Res_k  \|_2}{\|\Res_0\|_2} 
    = \frac{\| \vec{r}_{k}^M \|_2/\| \vec{r}_{0}^M\|_{2}}{\sqrt{1- \left( \| \vec{r}_{k}^M \|_2 / \| \vec{r}_{k-1}^M \|_2 \right)^2}}.
\end{align*}
\end{theorem}
\iffalse
\Cref{thm:minres_CG_residuals1} shows that at steps \( k \) where the MINRES residual is significantly reduced (\( \| \vec{r}_k^M \| / \| \vec{r}_{k-1}^M \| \ll 1 \)), the 2-norm of the Lanczos residual is very close to that of MINRES.
It follows from \cref{thm:minres_CG_residuals1} that the 2-norm of the error in the Lanczos approximation to the solution of a Hermitian linear system can be related to that in the optimal approximation from the Krylov space as follows:
\begin{align}
\label{eqn:indef_linear_systems}
\| \err_k \| \leq \kappa (\vec{A}) \frac{1}{\sqrt{1- \| \vec{r}_k^M \|_2 / \| \vec{r}_{k-1}^M \|_2}} \min_{\vec{y} \in \mathcal{K}_k ( \vec{A}, \vec{b} )} \| \vec{A}^{-1} \vec{b} - \vec{y} \| .
\end{align}
\fi
Therefore, if MINRES makes good progress at step \( k \) (i.e. \( \| \vec{r}_k^M \|_2 / \| \vec{r}_{k-1}^M \|_2 \) is small), then \cref{thm:minres_CG_residuals1} implies \( \| \Res_k \|_2/\|\Res_0\|_2 \approx \| \vec{r}_k^M \|_2/\|\vec{r}_0^M\|_2 \).
%Loosely, we therefore expect that Lanczos-FA to converge similarly to MINRES in the sense that
Thus, since MINRES converges at a linear rate,
there will be iterations in which Lanczos-FA has nearly as good a residual norm as MINRES.
This is made precise by the following result.

\begin{corollary}
\label{thm:indefinite_CG}
Suppose $\Lambda(\vec{A})\subset [a,b]\cup[c,d]$, where $a < b < 0 < c < d$ with $b-a = d-c$, and define $\gamma = \sqrt{|ad|/|bc|}$.
Then, for any $\epsilon < \gamma/4$ there exsits \( k \leq 2 \gamma \log(\sqrt{2} \gamma/\epsilon) \) so that \( \| \Res_k \|_2 / \| \Res_0 \|_2 < \epsilon \).
\end{corollary}

\begin{proof}
If the eigenvalues of $\vec{A}$ lie in $[a,b]\cup[c,d]$, where $a < b < 0 < c < d$ and $b-a = d-c$, then as in \cite[Section 3.1]{greenbaum_97}, the optimality of MINRES implies
\begin{align*}
    \frac{\|\vec{r}_{j}^M\|_2}{\|\vec{r}_0^M\|_2}
    \leq 2 \left( \frac{\sqrt{|ad|/|bc|}-1}{\sqrt{|ad|/|bc|}+1} \right)^{\lfloor j/2 \rfloor}
    \leq 2 \exp \left( - \frac{2 \lfloor j/2 \rfloor}{\sqrt{|ad|/|bc|}} \right).
\end{align*}

    For notational convenience set $\tau = 2\epsilon/\gamma$ and define $k'$ to be the first iteration where $\|r_{k'}^M\|_2/\|r_0\|_2 < \tau/2$ and $k''$ to be the first iteration where $\|r_{k''}^M\|_2/\|r_0\|_2 < \tau^{2}/4$.
Note that $k'' \leq \gamma \log(2/(\tau^2/4)) = 2 \gamma \log(2\sqrt{2}/\tau)$.

    First, suppose $\|r_{k'}\|_2/\|r_0\|_2 \leq \tau/4$. 
    Then, since $\|r_{k'-1}\|_2/\|r_0\|_2 > \tau/2$, using \cref{thm:minres_CG_residuals1},
\begin{align*}
    \frac{\| \Res_k  \|_2}{\|\Res_0\|_2} 
    = \frac{\| \vec{r}_{k'}^M \|_2/\| \vec{r}_{0}^M\|_2}{\sqrt{1- \left( \| \vec{r}_{k'}^M \|_2 / \| \vec{r}_{k'-1}^M \|_2 \right)^2}}
    \leq \frac{\tau/4}{\sqrt{1-((\tau/4)/(\tau/2))^2}}
    = \frac{\tau}{2\sqrt{3}}
    \leq \epsilon.
\end{align*}

    Next, suppose that $\|r_{k'}\|_2/\|r_0\|_2 > \tau/4$.
Let $\ell = k''-k$ and note that there must exist an iteration $k\in(k',k'']$ so that
\begin{align*}
    \frac{\| \vec{r}_k^M \|_2}{\| \vec{r}_{k-1}^M \|_2} 
    = \frac{\| \vec{r}_k^M \|_2/\| \vec{r}_0^M \|_2}{\| \vec{r}_{k-1}^M \|_2/\| \vec{r}_0^M \|_2} 
    \leq \left( \frac{\tau^2/4}{\tau/4} \right)^{1/\ell}.
\end{align*}
Now note that 
\begin{align*}
    \frac{1}{\sqrt{1- \left( \left(\frac{\tau^2/4}{\tau/4}\right)^{1/\ell} \right)^2}}
    = \frac{1}{\sqrt{1- \tau^{2/\ell}}}
\end{align*}
and that $\ell \leq k'' - 1 \leq 2 \gamma \log( 2\sqrt{2}/\tau)$ so
\begin{align*}
    \frac{1}{\sqrt{1- \tau^{2/\ell}}}
    \leq \frac{1}{\sqrt{1- \tau^{1/(\gamma \log( 2\sqrt{2}/\tau))}}}
    = \frac{1}{\sqrt{1- \left(\tau^{1/\log( 2\sqrt{2}/\tau)}\right)^{1/\gamma }}}
\end{align*}
If $\tau\in[0,1/2]$ then $\tau^{1/\log(2\sqrt{2}/\tau)} \leq \exp(-2/5) < 3/4$ so noting that $\gamma\geq1$ we can apply \cref{thm:sqrtxy_bound} to obtain
\begin{align*}
    \frac{1}{\sqrt{1- \left(\tau^{1/\log( 2\sqrt{2}/\tau)}\right)^{1/\gamma }}}
    \leq 2\gamma.
\end{align*}
Combining this with \cref{thm:minres_CG_residuals1} gives,
\begin{align*}
    \frac{\| \Res_k  \|_2}{\|\Res_0\|_2} 
    = \frac{\| \vec{r}_{k'}^M \|_2/\| \vec{r}_{0}^M\|_2}{\sqrt{1- \left( \| \vec{r}_{k'}^M \|_2 / \| \vec{r}_{k'-1}^M \|_2 \right)^2}}
    \leq \frac{\tau/4}{\sqrt{1- \left( \left(\frac{\tau^2/4}{\tau/4}\right)^{1/\ell} \right)^2}}
    \leq \frac{\tau \gamma}{2} = \epsilon.
\end{align*}
\end{proof}

\begin{lemma}
\label{thm:sqrtxy_bound}
For all $x\in[0,3/4]$ and $y\in[0,1]$,
\begin{align*}
    \frac{1}{\sqrt{1-x^y}} \leq \frac{2}{y} .
\end{align*}
\end{lemma}
\begin{proof}
Consider the function
\begin{align*}
    g(x,y) = \frac{y}{2\sqrt{1-x^y}}.
\end{align*}
For any $y\in[0,1]$, $g(x,y)$ is non-decreasing in $x$, so it suffices to set $x=3/4$.
Thus, define
\begin{align*}
    f(y) = \log(g(3/4,y)) = \log \left( \frac{y}{2\sqrt{1-(3/4)^y}} \right)
\end{align*}
which has derivative
\begin{align*}
    f'(y) = \frac{1}{y} - \frac{\log(4/3)}{2((4/3)^y-1)}.
\end{align*}?
Note that $(4/3)^y-1 \geq \log(4/3)y$ for all $y\geq0$ so 
\begin{align*}
    \frac{\log(4/3)}{2((4/3)^y-1)}
    \leq \frac{\log(4/3)}{2 \log(4/3)y}
    = \frac{1}{2y}.
\end{align*}
Therefore $f'(y) \geq 1/(2y) \geq 0$, so $f(y)$ is non-decreasing. 
Since $\log$ is increasing this implies that $g(3/4,y)$ is a non-decreasing function of $y$ on $[0,1]$ and therefore bounded above by $g(3/4,1) = 1$.
Thus, $g(x,y) \leq 1$ for all $x\in[0,3/4]$ and $y\in[0,1]$ and the result follows.
\end{proof}

So far we have discussed a priori bounds, but there are a range of a posteriori bounds as well.
For instance, a simple a posteriori bound is obtained using the fact that \( \| \err_k \|_{\vec{A}^2} = \| \Res_k    \|_2 \), which holds even when \( \vec{A}\) is indefinite.
Using the similarity of matrix norms, bounds for \( \| \err_k \| \) when \( \| \cdot \| \) is any norm induced by a matrix with the same eigenvectors as \( \vec{A} \) can then be obtained.

When \( \vec{A} \) is positive (or negative) definite, a range of more refined error bounds and estimates for the \( \vec{A} \)-norm and 2-norm have been considered.
These bounds obtain error estimates for CG at step $k$ by running Lanczos (or CG) for an extra $d$ iterations.
The information from this larger Krylov subspace $\mathcal{K}_{k+d}(\vec{A},\vec{b})$ is then used to estimate the error at step $k$.
Typically $d$ can be taken as a small constant, say $d=5$, so the extra work required to obtain these bounds is not too large.
We refer to \cite{strakos_tichy_02,meurant_tichy_18,estrin_orban_saunders_19,meurant_papez_tichy_21} and the references within for more details.

\bibliography{lanczos_error}
\bibliographystyle{siam}

\end{document}